\documentclass[10pt]{article}
\usepackage{amssymb,amsmath,amsthm,indentfirst}
\usepackage{pgfplotstable,booktabs,colortbl}
\usepackage{graphics,epsfig}
\usepackage{multicol}
\def\bh{\bf h}
\def\bPsi{\boldsymbol\Psi}
\def\bPhi{\boldsymbol\Phi}
\DeclareMathOperator{\Rea}{Re}
\DeclareMathOperator{\Ima}{Im}
\renewcommand{\leq}{\leqslant}
\renewcommand{\geq}{\geqslant}

\newtheorem{proposition}{Proposition}

\newtheorem{corollary}{Corollary}
\newtheorem{remark}{Remark}
\begin{document}
\large
\begin{center}
{\LARGE{\bf\sf The splitting in potential Crank-Nicolson scheme}}
\end{center}
\begin{center}
{\LARGE{\bf\sf with discrete transparent boundary conditions}}
\end{center}
\begin{center}
{\LARGE{\bf\sf for the Schr\"odinger equation on a semi-infinite strip}}
\vskip 0.5cm
{\large Bernard Ducomet,
\footnote{CEA, DAM, DIF, F-91297, Arpajon, France.
E-mail: \it{bernard.ducomet@cea.fr}}
{\large Alexander Zlotnik}
\footnote{Department of Higher Mathematics at Faculty of Economics,
National Research University Higher School of Economics,
Myasnitskaya 20, 101000 Moscow, Russia.
E-mail: \it{azlotnik2008@gmail.com}}
{\large  and Ilya Zlotnik}
\footnote{Department of Mathematical Modelling,
National Research University Moscow Power Engineering Institute,
Krasnokazarmennaya 14, 111250 Moscow, Russia.
E-mail: \it{ilya.zlotnik@gmail.com}}}
\end{center}
\vskip 0.5cm
\begin{abstract}
\noindent We consider an initial-boundary value problem for a generalized 2D time-dependent Schr\"odinger equation (with variable coefficients) on a semi-infinite strip. For the Crank-Nicolson-type finite-difference scheme with approximate or discrete transparent boundary conditions (TBCs), the Strang-type splitting with respect to the potential is applied. For the resulting method, the unconditional uniform in time $L^2$-stability is proved.
Due to the splitting, an effective direct algorithm using FFT is developed now to implement the method with the discrete TBC for general potential.
Numerical results on the tunnel effect for rectangular barriers are included together with the detailed practical error analysis confirming nice properties of the method.
\end{abstract}
\section{Introduction}
\label{sect1}
\large

\par The multidimensional time-dependent Schr\"odinger equation describes most of microscopic phenomena in non-relativistic quantum mechanics, atomic and nuclear physics and it also appears more generally in wave physics and nanotechnologies. Due to the physical framework (quantum mechanics) the corresponding initial value problem must be solved in unbounded space domains however, due to computational constraints, it is necessary to restrict the analysis to a bounded region which implies to solve the delicate problem of prescribing suitable boundary conditions.

\par In this work an additional complication is that we consider a generalized 2D time-dependent Schr\"odinger equation (GSE) with variable coefficients (when the Laplace operator is replaced by an operator of ``Laplace-Beltrami" type \cite{H74,RHR78,RDNPS78}) on a semi-infinite strip. This model appears in nuclear physics in the so-called Generator Coordinate Method (GCM) \cite{RS80} when one wants to describe microscopically large collective motions of nuclei and specifically low-energy nuclear fission dynamics \cite{BGG84,CBGW84,BGG91,GBCG05}.
%
\par Several approaches have been developed in order to solve
numerically such problems (see in particular, in the constant coefficients case \cite{AABES08,ABM04,AES03,DiM97,DZ06,DZ07,Sch02a,SA08,SZ04}).
One of them exploits the so-called discrete transparent boundary conditions (TBCs) on artificial boundaries \cite{EA01}.
Its advantages are the complete absence of spurious reflections, reliable computational stability, clear mathematical background and rigorous stability theory.

\par Concerning the discretization of the GSE, the Crank-Nicolson finite-difference scheme and more general schemes with the discrete TBCs in the case of a strip or semi-infinite strip was studied in detail in \cite{AES03,DZ06,DZ07,ZZ11,IZ11}.
However the scheme is implicit and solving a specific complex system of linear algebraic equations is required at each time level. In fact efficient methods to solve such systems are well developed by now in the real situation but not in the complex one.
Only the particular case of all the coefficients (including the potential) independent of the coordinate $y$ perpendicular to the strip can be effectively implemented \cite{ZZ11,IZ11}.
On the other hand, the splitting technique has been widely used to simplify the resolution of the time-dependent Schr\"o\-dinger and related equations (see in particular \cite{BM00,GX11,G11,L08a,L08,NT09}).

\par Our goal in this work is to apply the Strang-type splitting with respect to the potential to the Crank-Nicolson scheme with a sufficiently general approximate TBC in the form of the Dirichlet-to-Neumann map. The resulting method is more easy to implement but more difficult to study.
\par Developing the technique from \cite{DZ06}, we prove its unconditional uniform in time $L^2$-stability and conservativeness under a condition on an operator $\mathcal{S}$ in the approximate TBC.
To construct the discrete TBC, we are obliged to consider the splitting scheme on an infinite mesh in the semi-infinite strip. Its uniform in time $L^2$-stability together with the mass conservation law are proved. We find that an operator
$\mathcal{S}_{\rm ref}$ in the discrete TBC is the same as for the original Crank-Nicolson scheme in \cite{DZ06}, and it satisfies above mentioned condition so that the uniform in time $L^2$-stability of the resulting method is guaranteed.
The non-local operator
$\mathcal{S}_{\rm ref}$ is written in terms of the discrete convolution in time and the discrete Fourier expansion in direction $y$ perpendicular to the strip.
\par Due to the splitting, an effective direct algorithm using FFT in $y$ is developed to implement the method with the discrete TBC for \textit{general potential} (while other coefficients are $y$-independent). The corresponding numerical results on the tunnel effect for rectangular barriers are presented together with the detailed practical error analysis in the uniform in time and $C$ and $L^2$ in space norms confirming the good error properties of the splitting scheme. This conclusion is very important since other splittings are able to deteriorate the error behavior essentially, in particular, see \cite{ZaZ98,ZaZ99,ZI06}.
\par Finally we just mention that the previous results can be rather easily generalized to the case of a multidimensional parallelepiped infinite or semi-infinite in one of the space directions.
Also the case of higher order in space splitting schemes with the discrete TBCs for the classical Schr\"o\-dinger equation has been quite recently covered by another technique in \cite{DZR13,ZR13}.
\section{The Schr\"odinger equation on a semi-infinite strip
 and the splitting in potential Crank-Nicolson scheme with an approximate TBC}
\label{sect4}
\setcounter{equation}{0}
\setcounter{proposition}{0}
\setcounter{theorem}{0}
\setcounter{lemma}{0}
\setcounter{corollary}{0}
\setcounter{remark}{0}
Let us consider the generalized 2D time-dependent Schr\"odinger equation
\begin{equation}
 i\hbar\rho D_t\psi
 =({\mathcal H_0}+V) \psi\ \
 \text{for}\ \ (x,y)\in\Omega,\ \ t>0,
\label{sch2d}
\end{equation}
where $\Omega:=(0,\infty)\times (0,Y)$ is a semi-infinite strip, involving the 2D Hamiltonian operator
\[
 {\mathcal H}_0\psi:=-\frac{\hbar^{\,2}}{2}
\left[
  D_x(B_{11}D_x\psi)
 +D_x(B_{12}D_y\psi)
 +D_y(B_{21}D_x\psi)
 +D_y(B_{22}D_y\psi)
 \right].
\]
The real coefficients $\rho(x,y)$, ${\mathbf B}=\{B_{pq}(x,y)\}_{p,q=1}^2$ and $V(x,y)$ (the potential) are such that $\rho(x,y)\geq\underline\rho>0$ in $\Omega$ and the matrix ${\mathbf B}$ is symmetric and positive definite uniformly in $\Omega$.
Also $i$ is the imaginary unit, $\hbar>0$ is a physical
constant, $D_t$, $D_x$ and $D_y$ are partial derivatives, and the unknown wave function $\psi=\psi(x,t)$ is complex-valued.
\par We impose the following boundary condition, the condition at infinity and initial condition
\begin{gather}
 \psi(\cdot,t)|_{\partial\Omega}=0,\ \
 \|\psi(x,\cdot,t)\|_{L^2(0,Y)}\to 0\ \
 \text{as}\ \ x\to +\infty,
 \ \ \text{for any}\ \ t>0,
\label{bc2d}\\[1mm]
 \psi|_{t=0}=\psi^0(x,y)\ \ \text{in}\ \ \Omega.
 \label{ic2d}
\end{gather}
We also assume that
\begin{gather}
 B_{11}(x,y)=B_{1\infty}>0,\ \
 B_{12}(x,y)=B_{21}(x,y)=0,\ \
 B_{22}(x,y)=B_{2\infty}>0,\ \
\nonumber\\[1mm]
 \rho(x,y)=\rho_{\infty}>0,\ \ V(x,y)=V_{\infty},\ \
 \psi^0(x,y)=0\ \ \text{on}\ \ \Omega\backslash\Omega_{X_0},
\label{cc2d}
\end{gather}
 for some $X_0>0$, where $\Omega_{X}:=(0,X)\times (0,Y)$.
It is well-known that solution to problem \eqref{sch2d}-\eqref{cc2d} satisfies a non-local integro-differential TBC for any $x=X\geq X_0$ (for example see \cite{DZ06}) which we do not reproduce here.
\par Introduce a non-uniform mesh $\overline{\omega}^{\,\tau}$ in $t$ on $[0,\infty)$ with nodes
$0=t_0<\dots<t_m<\dots$, $t_m\to\infty$ as $m\to\infty$,
and steps $\tau_m:=t_m-t_{m-1}$.
Let $t_{m-1/2}=\frac{t_{m-1}+t_m}{2}$ and
$\omega^\tau:=\overline{\omega}^{\,\tau}\backslash\,\{0\}$.
In the differential case, {\it the splitting in potential method} can be represented as follows: three problems are solved sequentially step by step in time
\begin{gather}
 i\hbar\rho D_t\breve{\psi}=(\Delta V)\breve{\psi}\ \ \text{on}\ \ \Omega\times(t_{m-1},t_{m-1/2}],\ \
 \breve{\psi}|_{t=t_{m-1}}=\psi|_{t=t_{m-1}};
\label{s11}\\[1mm]
 i\hbar\rho D_t\widetilde{\psi}=({\mathcal H}_0+\widetilde{V})\widetilde{\psi}\ \ \text{on}\ \ \Omega\times(t_{m-1},t_m],\ \
 \widetilde{\psi}|_{t=t_{m-1}}=\breve{\psi}|_{t=t_{m-1/2}},
\label{s13}\\[1mm]
 \widetilde{\psi}|_{\partial\Omega}=0,\ \
 \|\widetilde{\psi}(x,\cdot,t)\|_{L^2(0,Y)}\to 0\ \
 \text{as}\ \ x\to\infty,
 \ \ \text{for}\ \ t\in (t_{m-1},t_m];
\label{s14}\\[1mm]
 i\hbar\rho D_t\psi=(\Delta V)\psi\ \ \text{on}\ \ \Omega\times(t_{m-1/2},t_m],\ \
 \psi|_{t=t_{m-1/2}}=\widetilde{\psi}|_{t=t_m},
\label{s15}\\[1mm]
\psi|_{t=0}=\psi^0\ \ \text{in}\ \ \Omega,
\label{s17}
\end{gather}
for any $m\geq 1$, where $\Delta V:=V-\widetilde{V}$ and $\widetilde{V}(x)$ is an auxiliary potential satisfying
\begin{gather}
 \widetilde{V}(x)=V_{\infty}\ \ \text{on}\ \ [X_0,\infty).
\label{cc2da}
\end{gather}
In the simplest case, $\widetilde{V}(x)=V_{\infty}$. But, in particular, to generalize results to the case of a strip and different constant values $V_{\pm\infty}$ of $V(x,y)$ at $x\to\pm\infty$, it is necessary to take non-constant $\widetilde{V}$; see also Section \ref{sectnum} below.
\par The Cauchy problems \eqref{s11} and \eqref{s15} can be easily solved explicitly, in particular,
\begin{gather}
 \breve{\psi}|_{t=t_{m-1/2}}
 =\exp\Bigl\{-i\frac{\tau_m}{2\hbar\rho}\Delta V\Bigr\}\,
 \psi|_{t=t_{m-1}},\ \
 \psi|_{t=t_m}=\exp\Bigl\{-i\frac{\tau_m}{2\hbar\rho}\Delta V\Bigr\}\,
 \widetilde{\psi}|_{t=t_m}.
\label{s17}
\end{gather}
Equation in \eqref{s13} is the original equation \eqref{sch2d} simplified by substituting $\widetilde{V}$ for $V$. Also $\breve{\psi}$ and $\widetilde{\psi}$ are auxiliary functions and $\psi$ is the main unknown one. This is a version of the Strang-type splitting \cite{L08} (though the original Strang splitting \cite{S68} was suggested with respect to space derivatives for the 2D transport equation; note that sharp error bounds for the Strang splitting for the 2D heat equation can be found in \cite{ZI06}). The symmetrized three-step form of this splitting ensures its second order of approximation for $\psi|_{t=t_m}$ with respect to $\tau_m$.
\par We turn to the fully discrete case. Fix some $X>X_0$ and introduce a non-uniform mesh $\overline{\omega}_{h,\infty}$
in $x$ on $[0,\infty)$ with nodes $0=x_0<\dots <x_J=X<\dots$
and steps $h_j:=x_j-x_{j-1}$ such that
$x_{J-2}\geq X_0$
and $h_j=h\equiv h_J$
for $j\geq J$.
Let $\omega_{h,\infty}:=\overline{\omega}_{h,\infty}\backslash\,\{0\}$,
$\overline{\omega}_{h}:=\{x_j\}_{j=0}^J$, $\omega_{h}:=\overline{\omega}_{h}\backslash\,\{0,X\}$
and $h_{j+1/2}:=\frac{h_j+h_{j+1}}{2}$.
\par We define the backward, modified forward and central difference quotients as well as two mesh averaging operators in $x$
\begin{gather*}
 \overline{\partial}_xW_j:= \frac{W_j-W_{j-1}}{h_j},\ \
 \widehat{\partial}_xW_j:= \frac{W_{j+1}-W_{j}}{h_{j+1/2}},\ \
 \overset{\circ}{\partial}_x W_j:= \frac{W_{j+1}-W_{j-1}}{2h_{j+1/2}},
\\[1mm]
 \overline{s}_xW_j=\frac{W_{j-1}+W_j}{2},\ \
 \widehat{s}_xW_j:=\frac{h_jW_j+h_{j+1}W_{j+1}}{2h_{j+1/2}}.
\end{gather*}
We define two mesh counterparts
of the inner product in the complex space $L^2(0,X)$:
\[
 \left(U, W\right)_{\omega_h}:=
 \sum_{j=1}^{J-1}
 U_jW_j^*h_{j+1/2},\ \
 \left(U, W\right)_{\overline{\omega}_h}:=
 \left(U, W\right)_{\omega_h}
 +U_JW_J^*\,\frac{h}{2}
\]
and the associated mesh norms $\|\cdot\|_{\omega_h}$ and
$\|\cdot\|_{\overline{\omega}_h}$ (of course, for mesh functions respectively
defined on $\omega_h$ or defined on $\overline{\omega}_h$ and equal zero at $x_0=0$).
Hereafter $z^*$, $\Rea z$ and $\Ima z$  denote the complex conjugate, the real and
the imaginary parts of $z\in {\mathbb C}$.
The above averaging operators are related by an identity
\begin{equation}
 \left(\widehat{s}_xW,U\right)_{\omega_h}=\sum_{j=1}^JW_j(\overline{s}_xU_j^*)h_j
 -\frac{1}{2}\,(W_1U_0^*h_1+W_JU_J^*h_J).
\label{smid}
\end{equation}
\par We also introduce a non-uniform mesh $\overline{\omega}_{\delta}$ in $y$ on $[0,Y]$ with nodes $0=y_0<\dots <y_K=Y$ and steps $\delta_k:=y_k-y_{k-1}$.
Let $\omega_{\delta}:=\overline{\omega}_{\delta}\backslash\{0,Y\}$.
We define the backward and the modified forward difference quotients together with
two mesh averaging operators in $y$
\[
\overline{\partial}_yU_k:= \frac{U_k-U_{k-1}}{\delta_k},\ \
 \widehat{\partial}_yU_k:= \frac{U_{k+1}-U_{k}}{\delta_{k+1/2}},\ \
 \overline{s}_yU_k=\frac{U_{k-1}+U_k}{2},\ \
 \widehat{s}_yU_k:=\frac{\delta_kU_k+\delta_{k+1}U_{k+1}}{2\delta_{k+1/2}},
\]
where $\delta_{k+1/2}:=\frac{\delta_k+\delta_{k+1}}{2}$.
Let $\overset{\circ}{H}(\overline{\omega}_\delta)$ be the space of
functions $U$: $\overline{\omega}_\delta\to \mathbb{C}$ such that $U|_{k=0,K}=0$,
equipped with the inner product
\[
 (U,W)_{\omega_\delta}:=\sum_{k=1}^{K-1}U_kW_k^*\delta_{k+1/2}
\]
and the associated norm $\|\cdot\|_{\omega_\delta}$.
\par We define the product 2D meshes
$\overline{\omega}_{\bf h,\infty}
:=\overline{\omega}_{h,\infty}\times\overline{\omega}_{\delta}$
on $\overline\Omega$ and
$\overline{\omega}_{\bh}:=\overline{\omega}_h\times\overline{\omega}_{\delta}$
on $\bar{\Omega}_X$
as well as their interiors $\omega_{\bf h,\infty}:=\omega_{h,\infty}\times\omega_{\delta}$ and
$\omega_{\bh}:=\omega_h\times\omega_{\delta}$.
Let $\Gamma_{\bh}=\{(0,y_k),\,1\leq k\leq K-1\}\cup\{(x_j,0),(x_j,Y),\,0\leq j\leq J\}$ be a part of the boundary of $\overline{\omega}_{\bh}$.
\par Let $A_{-,\,jk}:=A(x_{j-1/2},y_{k-1/2})$, for all the coefficients $A=\rho,B_{pq},V$, with $x_{j-1/2}:=\frac{x_{j-1}+x_j}{2}$ and
$y_{k-1/2}:=\frac{y_{k-1}+y_k}{2}$.
We exploit the 2D mesh Hamiltonian operator
\[
 {\mathcal H}_{0\bf h} W:=-\frac{\hbar^{\,2}}{2}
 \left[
 \widehat{\partial}_x(B_{11\bh}\overline{\partial}_xW)
 +\widehat{\partial}_x\widehat{s}_y(B_{12\bh}\overline{s}_x\overline{\partial}_yW)
 +\widehat{s}_x\widehat{\partial}_y(B_{21\bh}\overline{\partial}_x\overline{s}_yW)
 +\widehat{\partial}_y(B_{22\bh}\overline{\partial}_y W)\right],
\]
where the coefficients are given by formulas
${B_{11\bh}}=\widehat{s}_y B_{11,-}$,
${B_{22\bh}}=\widehat{s}_x B_{22,-}$,
${B_{12\bh}}={B_{21h}}= B_{12,-}$.
We also set ${\rho_{\bh}}=\widehat{s}_x\widehat{s}_y\rho_{-}$,
$V_{\bh}=\widehat{s}_x\widehat{s}_yV_{-}$ and $\widetilde{V}_h=\widehat{s}_x\widetilde{V}_{-}$.
Actually this finite-difference discretization is a simplification of the bilinear finite element method for the rectangular mesh $\overline{\omega}_{\bh}$ (conserving, in particular, its
$L^2(\Omega)$ and $H^1(\Omega)$ optimal error bounds), see \cite{Z91}. Some other operators ${\mathcal H}_{0\bf h}$ could be also exploited.
\par We define also the backward difference quotient and an averaging in time
\[
 \overline{\partial}_t\Phi^m:= \frac{\Phi^m-\Phi^{m-1}}{\tau_m},\ \
 \overline{s}_t\Phi^m:= \frac{\Phi^{m-1}+\Phi^m}{2}.
\]
\par The following Crank-Nicolson-type scheme was studied in \cite{DZ06}
\begin{gather}
 i\hbar{\rho_{\bh}}\overline{\partial}_t\Psi^m
 =({\mathcal H}_{0\bf h}+V_{\bh}) \overline{s}_t\Psi^m\ \
 \text{on}\ \ \omega_{\bh},
\label{dse2d}\\[1mm]
 \Psi^m|_{\Gamma_{\bh}}=0,
\label{ds02d}\\[1mm]
 \left.\left\{
 \frac{\hbar^{\,2}}{2}B_{1\infty}\overline{\partial}_x \overline{s}_t\Psi
 -\frac{h}{2}
 \left[i\hbar{\rho_{\infty}}\overline{\partial}_t \Psi
 +\Bigl(\frac{\hbar^{\,2}}{2}B_{2\infty}\widehat{\partial}_y\overline{\partial}_y
  -V_{\infty}\Bigr)\overline{s}_t\Psi\right]
 \right\}^m\right|_{j=J}
 =\frac{\hbar^{\,2}}{2}B_{1\infty}{\mathcal S}^m\bPsi^m_J,
\label{dsJ2d}\\
 \Psi^0=\Psi^0_{\bh}\ \ \text{on}\ \ \overline{\omega}_{\bh},
\label{dsi2d}
\end{gather}
for any $m\geq 1$. Here the boundary condition \eqref{dsJ2d} is the general approximate TBC posed on $\omega_\delta$, with a linear operator ${\mathcal S}^m$ acting in the space of functions defined on $\omega_\delta\times\{t_l\}_{l=1}^m$, and $\bPsi^m_J=\{\Psi^1_{J\cdot},\ldots,\Psi^m_{J\cdot}\}$.
Also ${\Psi^0_{\bh}}_{jk}=\psi^0(x_j,y_k)$ (for definiteness) and thus
$\left.\Psi^0_{\bh}\right|_{j=J}=0$; we assume also that $\left.\Psi^0_{\bh}\right|_{j=J-1}=0$ and
the conjunction condition $\Psi^0_{\bh}|_{\Gamma_{\bh}}=0$ is valid.
\par Recall that the left-hand side in the approximate TBC \eqref{dsJ2d} has the form of the well-known 2D second order approximation to $\frac{\hbar^{\,2}}{2}B_{1\infty}D_x$ in the Neumann boundary condition (exploiting an $8$-point stencil in all the directions $x$, $y$ and $t$).
\par We write down the following Strang-type splitting in potential for the  Crank-Nicolson scheme \eqref{dse2d}-\eqref{dsi2d}
\begin{gather}
 i\hbar{\rho_{\bh}}\,\frac{\breve{\Psi}^m-\Psi^{m-1}}{\tau_m/2}
 =\Delta V_{\bh}\frac{\breve{\Psi}^m+\Psi^{m-1}}{2}\ \
 \text{on}\ \ (\omega_h\cup{x_J})\times\omega_{\delta},
\label{sd11}\\[1mm]
 i\hbar{\rho_{\bh}}\frac{\widetilde{\Psi}^m-\breve{\Psi}^m}{\tau_m}
 =(\mathcal{H}_{0{\bh}}+\widetilde{V}_h)\frac{\widetilde{\Psi}^m+\breve{\Psi}^m}{2}+F^m
 \ \ \text{on}\ \ \omega_{\bh},
\label{sd13}\\[1mm]
 i\hbar{\rho_{\bh}}\,\frac{\Psi^m-\widetilde{\Psi}^m}{\tau_m/2}
 =\Delta V_{\bh}\frac{\Psi^m+\widetilde{\Psi}^m}{2}\ \
  \text{on}\ \ (\omega_h\cup{x_J})\times\omega_{\delta},
\label{sd15}\\[1mm]
 \breve{\Psi}^m|_{\Gamma_{\bh}}=0,\ \ \widetilde{\Psi}^m|_{\Gamma_{\bh}}=0,\ \
 \Psi^m|_{\Gamma_{\bh}}=0,
\label{sd17}\\[1mm]
 \left.
 \left\{
 \frac{\hbar^{\,2}}{2}B_{1\infty}\overline{\partial}_x
 \frac{\widetilde{\Psi}^m+\breve{\Psi}^m}{2}
 -\frac{h}{2}
 \left[i\hbar{\rho_{\infty}}\frac{\widetilde{\Psi}^m-\breve{\Psi}^m}{\tau_m}
 +\Bigl(\frac{\hbar^{\,2}}{2}B_{2\infty}\widehat{\partial}_y\overline{\partial}_y
 -V_{\infty}\Bigr)\frac{\widetilde{\Psi}^m+\breve{\Psi}^m}{2}\right]
 \right\}\right|_{j=J}
 \nonumber\\[1mm]
 +\frac{h}{2}F^m|_{j=J}=\frac{\hbar^{\,2}}{2}B_{1\infty}{\mathcal S}^m\widetilde{\bPsi}^m_J
 \ \ \text{on}\ \ \omega_\delta,
\label{sd19}\\[1mm]
 \Psi^0=\Psi^0_{\bh}\ \ \text{on}\ \ \overline{\omega}_{\bh},
\label{sd21}
\end{gather}
for any $m\geq 1$, where $\Delta V_{\bh}:=V_{\bh}-\widetilde{V}_h$.
We have added the perturbation $F^m$ into \eqref{sd13} and \eqref{sd19} in order to study stability of the scheme below in more detail (we suppose that $F^m$ is given on $\overline{\omega}_{\bh}$ and $F^m|_{\Gamma_{\bh}}=0$).
\par Obviously equations \eqref{sd11} and \eqref{sd15}
are reduced to the explicit expressions
\begin{equation}
 \breve{\Psi}^m=\mathcal{E}^m\Psi^{m-1},\ \
 \Psi^m=\mathcal{E}^m\widetilde{\Psi}^m,\ \ \text{with}\ \
 \mathcal{E}^m
 :=\frac{1-i\displaystyle{\frac{\tau_m}{4\hbar\rho_{\bh}}}\Delta V_{\bh}}
        {1+i\displaystyle{\frac{\tau_m}{4\hbar\rho_{\bh}}}\Delta V_{\bh}},\ \ \text{on}\ \ (\omega_h\cup{x_J})\times\omega_{\delta}.
\label{sd23}
\end{equation}
The main finite-difference equation \eqref{sd13} is similar to the original one \eqref{dse2d} simplified by substituting $\widetilde{V}_h$ for $V_{\bh}$. Here $\breve{\Psi}$ and $\widetilde{\Psi}$ are auxiliary unknown functions and $\Psi$ is the main unknown one.
We have got the approximate TBC \eqref{sd19} by
substituting respectively
\[
 \frac{\widetilde{\Psi}^m+\breve{\Psi}^m}{2},\ \
 \frac{\widetilde{\Psi}^m-\breve{\Psi}^m}{\tau_m},\ \
 \widetilde{\bPsi}^m_J
\]
for $\overline{s}_t\Psi$, $\overline{\partial}_t \Psi$ and $\bPsi^m_J$  in the approximate TBC \eqref{dsJ2d}; but notice that since $\Delta V_{\bh}|_{j=J}=0$, actually
\begin{equation}
 \breve{\Psi}^m_{J\cdot}=\Psi^{m-1}_{J\cdot},\ \
 \Psi^m_{J\cdot}=\widetilde{\Psi}^m_{J\cdot}
 \ \ \text{on}\ \ \overline{\omega}_\delta\ \ \text{for}\ \ m\geq 1.
\label{sd24}
\end{equation}
\par Clearly the constructed splitting in potential scheme can be also considered as the Crank-Nicolson-type discretization in time and the same approximation in space for the above splitting in potential differential problem \eqref{s11}-\eqref{s17}.
\par Also note that inserting formulas \eqref{sd23} into equation \eqref{sd13} (for $F=0$) and excluding the auxiliary functions leads to the following equation for $\Psi$
\begin{equation}
 i\hbar{\rho_{\bh}}\,\frac{(\mathcal{E}^m)^{-1}\Psi^m-\mathcal{E}^m\Psi^{m-1}}{\tau_m}
 =(\mathcal{H}_{0{\bh}}+\widetilde{V}_h)
 \frac{\mathcal{E}^m\Psi^{m-1}+(\mathcal{E}^m)^{-1}\Psi^m}{2}
\label{sd25}
\end{equation}
or, in another form,
\begin{equation}
 \left[i\hbar{\rho_{\bh}}-\frac{\tau_m}{2}(\mathcal{H}_{0{\bh}}+\widetilde{V}_h)\right]
 (\mathcal{E}^m)^{-1}\Psi^m
 =\left[i\hbar{\rho_{\bh}}+\frac{\tau_m}{2}(\mathcal{H}_{0{\bh}}+\widetilde{V}_h)\right]
 \mathcal{E}^m\Psi^{m-1}.
\label{sd26}
\end{equation}
Equation \eqref{sd25} can be considered as a non-standard discretization for the Schr\"odinger equation \eqref{sch2d} whereas equation \eqref{sd26} can be viewed as a specific symmetric approximate factorization \cite{Ya71} with respect to the potential of the Crank-Nicolson equation \eqref{dse2d}.
\par We note that ${\mathcal E}^{-1}={\mathcal E}^*$ and write down
 ${\mathcal E}={\mathcal E}_R-2\tau i\hat{{\mathcal E}}_I$ with real ${\mathcal E}_R$ and $\hat{{\mathcal E}}_I$. Then
\begin{gather*}
 \frac{{\mathcal E}^*\Psi-{\mathcal E} \check{\Psi}}{\tau}
 ={\mathcal E}_R\bar{\partial}_t \Psi
 +4i\hat{{\mathcal E}}_I\overline{s}_t \Psi,
\ \
 \frac{{\mathcal E}^*\Psi+{\mathcal E} \check{\Psi}}{2}
 ={\mathcal E}_R\overline{s}_t \Psi
 +i\tau^2\hat{{\mathcal E}}_I\bar{\partial}_t \Psi.
\end{gather*}
Consequently we can rewrite equation (\ref{sd25}) as follows
\begin{gather*}
 i\hbar{\rho_{\bh}} {\mathcal E}_R\bar{\partial}_t \Psi
 =\mathcal{H}_{0{\bh}}({\mathcal E}_R\overline{s}_t \Psi
+i\tau^2\hat{{\mathcal E}}_I\bar{\partial}_t \Psi)
+({\mathcal E}_R\widetilde{V}_h+4\hbar{\rho_{\bh}}\hat{{\mathcal E}}_I)\overline{s}_t \Psi
+i\tau^2\hat{{\mathcal E}}_I\widetilde{V}_h\bar{\partial}_t \Psi.
\end{gather*}
After some calculations, this formulation implies that the approximation error of the discrete equation (\ref{sd25}) differs from the original one (\ref{dse2d}) by a term of the order $O\left(\tau_{\max}^2\right)$
(in particular, note that
${\mathcal E}_R=1+O(\tau_{\max}^2)$ and $4\hbar{\rho_{\bh}}\hat{{\mathcal E}}_I=\Delta V_{\bh}+O(\tau_{\max}^2)$) as $\tau_{\max}\to 0$.

\par Since the Cauchy problems \eqref{s11} and \eqref{s15} do not need necessarily a discretization in time,
to cover both formulas \eqref{s17} and \eqref{sd23}, below we also admit an expression
\begin{equation}
 \mathcal{E}^m
 =\exp\Bigl\{-i\frac{\tau_m}{2\hbar\rho_{\bh}}\Delta V_{\bh}\Bigr\}
\label{sd27}
\end{equation}
in \eqref{sd23}. Obviously in both cases $\mathcal{E}^{-1}=\mathcal{E}^*$ and $|\mathcal{E}|=1$.
\par We introduce two mesh counterparts
of the inner product in the complex space $L^2(\Omega_X)$:
\[
 \left(U, W\right)_{\omega_{\bh}}:=
 \sum_{j=1}^{J-1}\sum_{k=1}^{K-1}
 U_{jk}W_{jk}^*h_{j+1/2}\delta_{k+1/2},\ \
 \left(U, W\right)_{\overline{\omega}_{\bh}}:=
 \left(U, W\right)_{\omega_{\bh}}
 +\sum_{k=1}^{K-1}U_{Jk}W_{Jk}^*\frac{h}{2}\,\delta_{k+1/2}
\]
and the associated mesh norms $\|\cdot\|_{\omega_{\bh}}$ and
$\|\cdot\|_{\overline{\omega}_{\bh}}$.
\begin{proposition}
\label{p12d}
Let the operator ${\mathcal S}$ satisfy an inequality \cite{DZ06}
\begin{equation}
 \Ima\sum_{m=1}^M
 \left({\mathcal S}^m\bPhi^m,
 \overline{s}_t\Phi^{m}\right)_{\omega_{\delta}}\tau_m
 \geq 0\ \ \text{for any}\ M\geq 1,
\label{cs2d}
\end{equation}
for any function $\Phi$: $\overline{\omega}_\delta\times\overline{\omega}^{\,\tau}\to \mathbb{C}$ such that $\Phi^0=0$ and
$\left.\Phi\right|_{k=0,K}=0$, where $\bPhi^m=\{\Phi^1,\ldots,\Phi^m\}$.
Then, for a solution of the splitting in potential scheme \eqref{sd11}-\eqref{sd21},
the following stability bound holds
\begin{equation}
 \max_{0\leq m\leq M} \|\sqrt{\rho_{\bh}}\Psi^m\|_{\overline{\omega}_{\bh}}
 \leq
 \|\sqrt{\rho_{\bh}}\Psi_{\bh}^0\|_{\overline{\omega}_{\bh}}
 +\frac{2}{\hbar}\sum_{m=1}^M
 \left\|\frac{F^m}{\sqrt{\rho_{\bh}}}\right\|_{\overline{\omega}_{\bh}}\tau_m
 \ \ \text{for any}\ M\geq 1.
\label{sb2d}
\end{equation}
\end{proposition}
\begin{proof}
We take
the $(\cdot, \cdot)_{\omega_{\bh}}$-inner-product of equation \eqref{sd13}
with a function $W$: $\overline{\omega}_{\bh}\to \mathbb{C}$ such that
$W|_{\Gamma_{\bh}}=0$. Then we sum the result by parts in $x$ and $y$
(using assumptions \eqref{cc2d} and \eqref{cc2da}), apply identity \eqref{smid} and a similar identity
with respect to $y$, exploit the approximate TBC \eqref{sd19} and
obtain an identity
\begin{gather}
i\hbar
 \Bigl(\rho_{\bh}\frac{\widetilde{\Psi}^m-\breve{\Psi}^m}{\tau_m}, W\Bigr)_{\overline{\omega}_{\bh}}
\nonumber\\[1mm]
 =\frac{\hbar^{\,2}}{2}\sum_{j=1}^J\sum_{k=1}^K
 \Bigl\{
  B_{11,-}
  \overline{s}_y\Bigl[\Bigl(\overline{\partial}_x
  \frac{\widetilde{\Psi}^m+\breve{\Psi}^m}{2}\Bigr)
  \overline{\partial}_xW^*
  \Bigr]
 +B_{12,-}\Bigl(\overline{s}_x\overline{\partial}_y
 \frac{\widetilde{\Psi}^m+\breve{\Psi}^m}{2}\Bigr)
 \overline{\partial}_x\overline{s}_yW^*
 \Bigr.
\nonumber
\end{gather}
\begin{gather}
\Bigl.
  +B_{21,-}\Bigl(\overline{\partial}_x\overline{s}_y
  \frac{\widetilde{\Psi}^m+\breve{\Psi}^m}{2}\Bigr)
  \overline{s}_x\overline{\partial}_yW^*
  +B_{22,-}\overline{s}_x
  \Bigl[\Bigl(\overline{\partial}_y
  \frac{\widetilde{\Psi}^m+\breve{\Psi}^m}{2}\Bigr)
  \overline{\partial}_yW^*\Bigr]
  \Bigr\}_{jk}h_j\delta_k
\nonumber\\[1mm]
 +\Bigl(\widetilde{V}_h\frac{\widetilde{\Psi}^m+\breve{\Psi}^m}{2}, W\Bigr)_{\overline{\omega}_{\bh}}
 +(F^m, W)_{\overline{\omega}_{\bh}}
 -\frac{\hbar^{\,2}}{2}B_{1\infty}
 \Bigl({\mathcal S}^m\widetilde{\bPsi}^m_J, W_{J\cdot}\Bigr)_{\omega_\delta}
\label{sumidm2d}
\end{gather}
for any $m\geq 1$, see \cite{DZ06} for more details.
\par The sesquilinear form on the right-hand side containing the  five terms with coefficients $\widetilde{B}_{pq}$ and $\widetilde{V}_h$ is Hermitian-symmetric. Thus choosing $W=\frac{\widetilde{\Psi}^m+\breve{\Psi}^m}{2}$ and separating the imaginary part of the result, we get
\begin{gather*}
 \frac{\hbar}{2\tau_m}
 \Bigl[\Bigl(\rho_{\bh}\widetilde{\Psi}^m,\widetilde{\Psi}^m\Bigr)_{\overline{\omega}_{\bh}}
 -\Bigl(\rho_{\bh}\breve{\Psi}^m,\breve{\Psi}^m\Bigr)_{\overline{\omega}_{\bh}}\Bigr]
\\[1mm]
 =\Ima \Bigl(F^m,\frac{\widetilde{\Psi}^m+\breve{\Psi}^m}{2}\Bigr)_{\overline{\omega}_{\bh}}
 -\frac{\hbar^{\,2}}{2}B_{1\infty}\Ima \Bigl({\mathcal S}^m\widetilde{\bPsi}^m_J,
 \frac{\widetilde{\Psi}^m_{J\cdot}+\breve{\Psi}^m_{J\cdot}}{2}\Bigr)_{\omega_\delta}.
\end{gather*}
Owing to \eqref{sd17}, \eqref{sd23} and \eqref{sd27} we have the pointwise equalities
\begin{equation}
 |\breve{\Psi}^m|=|\Psi^{m-1}|,\ \ |\Psi^m|=|\widetilde{\Psi}^m|\ \ \text{on}\ \ \overline{\omega}_{\bh}.
\label{p409}
\end{equation}
Also taking into account equalities \eqref{sd24}, we further derive
\begin{gather*}
 \frac{\hbar}{2\tau_m}
 \left(\|\sqrt{\rho_{\bh}}\Psi^m\|_{\overline{\omega}_{\bh}}^2
 -\|\sqrt{\rho_{\bh}}\Psi^{m-1}\|_{\overline{\omega}_{\bh}}^2\right)
\nonumber\\[1mm]
 =\Ima \Bigl(F^m,\frac{\widetilde{\Psi}^m+\breve{\Psi}^m}{2}\Bigr)_{\overline{\omega}_{\bh}}
 -\frac{\hbar^{\,2}}{2}B_{1\infty}\Ima \left({\mathcal S}^m\bPsi^m_J,
 \overline{s}_t\Psi_{J\cdot}^m\right)_{\omega_\delta}.
\end{gather*}
Multiplying both sides by $\frac{2\tau_m}{\hbar}$ and summing up the result over $m=1,\ldots,M$, we obtain
\begin{gather}
 \bigl\|\sqrt{\rho_{\bh}}\Psi^M\bigr\|_{\overline{\omega}_{\bh}}^2
 +\hbar B_{1\infty}\sum_{m=1}^M\Ima \left({\mathcal S}^m\bPsi^m_J,
 \overline{s}_t\Psi^m_{J\cdot}\right)_{\omega_\delta}\tau_m
\nonumber\\[1mm]
 =\bigl\|\sqrt{\rho_{\bh}}\Psi^0\bigr\|_{\overline{\omega}_{\bh}}^2
 +\frac{2}{\hbar}\sum_{m=1}^M\Ima \Bigl(F^m,\frac{\widetilde{\Psi}^m+\breve{\Psi}^m}{2}\Bigr)_{\overline{\omega}_{\bh}}\tau_m.
\label{p413}
\end{gather}
Applying inequality \eqref{cs2d} and then equalities \eqref{p409}, we get
\begin{gather*}
 \bigl\|\sqrt{\rho_{\bh}}\Psi^M\bigr\|_{\overline{\omega}_{\bh}}^2
 \leq\bigl\|\sqrt{\rho_{\bh}}\Psi^0\bigr\|_{\overline{\omega}_{\bh}}^2
\\[1mm]
 +\frac{1}{\hbar}\sum_{m=1}^M
 \left\|\frac{F^m}{\sqrt{\rho_{\bh}}}\right\|_{\overline{\omega}_{\bh}}\tau_m
 \Bigl(\max_{1\leq m\leq M}
 \bigl\|\sqrt{\rho_{\bh}}\widetilde{\Psi}^m\bigr\|_{\overline{\omega}_{\bh}}
 +\max_{1\leq m\leq M}
 \bigl\|\sqrt{\rho_{\bh}}\breve{\Psi}^m\bigr\|_{\overline{\omega}_{\bh}}\Bigr)
\\[1mm]
\leq\bigl\|\sqrt{\rho_{\bh}}\Psi^0\bigr\|_{\overline{\omega}_{\bh}}^2
 +\frac{2}{\hbar}\sum_{m=1}^M
 \left\|\frac{F^m}{\sqrt{\rho_{\bh}}}\right\|_{\overline{\omega}_{\bh}}\tau_m
 \max_{0\leq m\leq M}\bigl\|\sqrt{\rho_{\bh}}\Psi^m\bigr\|_{\overline{\omega}_{\bh}}.
\end{gather*}
This directly implies bound \eqref{sb2d}.
\end{proof}
\begin{corollary}
\label{cp12d}
Let condition \eqref{cs2d} be valid. Then the splitting in potential scheme \eqref{sd11}-\eqref{sd21}
is uniquely solvable.
\par In particular, for $F=0$ its solution satisfies an equality
\begin{equation}
 \max_{m\geq 0} \|\sqrt{\rho_{\bh}}\Psi^m\|_{\overline{\omega}_{\bh}}
 =\|\sqrt{\rho_{\bh}}\Psi_{\bh}^0\|_{\overline{\omega}_{\bh}}.
\label{eneq2d}
\end{equation}
\end{corollary}
\begin{proof}
The unique solvability follows from a priori bound \eqref{sb2d}, and equality \eqref{eneq2d} also is clear from \eqref{sb2d} for $F=0$.
\end{proof}
\begin{remark}
We emphasize that actually both the Crank-Nicolson scheme and the splitting in potential scheme can be similarly considered and studied not only for the strip geometry but for much more general unbounded domain $\Omega$ composed of rectangles with sides parallel to coordinate axes and having one or more separated semi-infinite strip outlets at infinity.
\end{remark}
\section{The splitting in potential Crank-Nicolson scheme on the infinite space mesh and the discrete TBC}
\label{sect5}
\setcounter{equation}{0}
\setcounter{proposition}{0}
\setcounter{theorem}{0}
\setcounter{lemma}{0}
\setcounter{corollary}{0}
\setcounter{remark}{0}
To construct the discrete TBC, it is required to consider the splitting in potential Crank-Nicolson scheme on the infinite space mesh
for the original problem \eqref{sch2d}-\eqref{cc2d} on the semi-infinite strip
\begin{gather}
 i\hbar{\rho_{\bh}}\,\frac{\breve{\Psi}^m-\Psi^{m-1}}{\tau_m/2}
 =\Delta V_{\bh}\frac{\breve{\Psi}^m+\Psi^{m-1}}{2}\ \
 \text{on}\ \ \omega_{{\bh},\infty},
\label{sdi11}\\[1mm]
 i\hbar{\rho_{\bh}}\frac{\widetilde{\Psi}^m-\breve{\Psi}^m}{\tau_m}
 =(\mathcal{H}_{0{\bh}}+\widetilde{V}_h)\frac{\widetilde{\Psi}^m+\breve{\Psi}^m}{2}+F^m
 \ \ \text{on}\ \ \omega_{{\bh},\infty},
\label{sdi13}\\[1mm]
 i\hbar{\rho_{\bh}}\,\frac{\Psi^m-\widetilde{\Psi}^m}{\tau_m/2}
 =\Delta V_{\bh}\frac{\Psi^m+\widetilde{\Psi}^m}{2}\ \
  \text{on}\ \ \omega_{{\bh},\infty},
\label{sdi15}\\[1mm]
 \breve{\Psi}^m|_{\Gamma_{\bh,\infty}}=0,\ \
 \widetilde{\Psi}^m|_{\Gamma_{\bh,\infty}}=0,\ \
 \Psi^m|_{\Gamma_{\bh,\infty}}=0,
\label{sdi17}\\[1mm]
 \Psi^0=\Psi^0_{\bh}\ \ \text{on}\ \ \overline{\omega}_{\bf h,\infty},
\label{sdi21}
\end{gather}
for any $m\geq 1$, where $\Gamma_{\bh,\infty}:=\overline{\omega}_{\bf h,\infty}\backslash\omega_{{\bh},\infty}$.
The perturbation $F^m$ is given on $\overline{\omega}_{{\bh},\infty}$ after setting $F^m|_{\Gamma_{\bh,\infty}}=0$; it is added to the right-hand side of the main equation \eqref{sdi13} once again to study stability in more detail.
\par Obviously once again equations \eqref{sdi11} and \eqref{sdi15} are reduced to explicit expressions \eqref{sd23} which are valid now on $\omega_{\bf h,\infty}$. We continue to admit expression \eqref{sd27} in \eqref{sd23}.
\par Let $H_{\bh}$ be a Hilbert space of mesh functions
$W$: $\overline{\omega}_{{\bh},\infty}\to\mathbb{C}$ such that
$W|_{\Gamma_{\bh,\infty}}=0$ and
$\sum_{j=1}^\infty \|W_{jk}\|_{\omega_\delta}^2<\infty$,
equipped with the inner product
\[
 (U,W)_{H_{\bh}}:=\sum_{j=1}^\infty\sum_{k=1}^{K-1} U_{jk}W_{jk}^*h_{j+1/2}\delta_{k+1/2}.
\]
\begin{proposition}
\label{p22d}
Let $F^m,\Psi_{\bh}^0\in H_{\bh}$ for any $m\geq 1$.
Then there exists a unique solution to the splitting in potential scheme \eqref{sdi11}-\eqref{sdi21}
such that $\Psi^m\in H_{\bh}$ for any $m\geq 0$, and the following stability bound holds
\begin{equation}
 \max_{0\leq m\leq M}\|\sqrt{\rho_{\bh}}\Psi^m\|_{H_{\bh}}
 \leq\|\sqrt{\rho_{\bh}}\Psi^0_{\bh}\|_{H_{\bh}}
 +\frac{2}{\hbar}\sum_{m=1}^M
 \left\|\frac{F^m}{\sqrt{\rho_{\bh}}}\right\|_{H_{\bh}}\tau_m \ \ \text{for any}\ \ M\geq 1.
\label{sb22d}
\end{equation}
\par Moreover, in the particular case $F=0$, the mass conservation law holds
\begin{equation}
 \|\sqrt{\rho_{\bh}}\Psi^m\|_{H_{\bh}}^2=
 \|\sqrt{\rho_{\bh}}\Psi^0_{\bh}\|_{H_{\bh}}^2
 \ \ \text{for any}\ \ m\geq 1.
\label{cl2d}
\end{equation}
\end{proposition}
\begin{proof}
Given $\breve{\Psi}^m,F^m\in H_{\bh}$, there exists a unique solution
$\widetilde{\Psi}^m\in H_{\bh}$ to equation \eqref{sdi13}. The much more general result was established in the proof of the corresponding Proposition 3 in \cite{IZ11}.
Since expressions \eqref{sd23} are valid now on $\omega_{\bf h,\infty}$, this implies existence of a unique solution to the splitting scheme \eqref{sdi11}-\eqref{sdi21}
such that $\Psi^m\in H_{\bh}$ for any $m\geq 0$.
\par We set $\overset{\circ}{\mathcal H}_{0\bh}W:={\mathcal H}_{0\bh}W$
on $\omega_{\bf h,\infty}$ and $\overset{\circ}{\mathcal H}_{0\bh}W:=0$
on $\Gamma_{\bh,\infty}$.
The operator $\overset{\circ}{\mathcal H}_{0\bh}$ is bounded and self-adjoint in $H_{0\bh}$ since
\begin{gather*}
\bigl(\overset{\circ}{\mathcal H}_{0\bh}U,W\bigr)_{H_{\bh}}
 = \frac{\hbar^{\,2}}{2}
 \sum_{j=1}^\infty\sum_{k=1}^K
 \left\{
  B_{11,-}\overline{s}_y\left[(\overline{\partial}_xU)\,
  \overline{\partial}_xW^*\right]
 +B_{12,-}(\overline{s}_x\overline{\partial}_yU)\,
 \overline{\partial}_x\overline{s}_yW^*
 \right.
\nonumber\\[1mm]
\left.
 +B_{21,-}(\overline{\partial}_x\overline{s}_yU)\,
 \overline{s}_x\overline{\partial}_yW^*
 +B_{22,-}\overline{s}_x\left[(\overline{\partial}_yU)\,
 \overline{\partial}_yW^*\right]
 \right\}_{jk}h_j\delta_k
\end{gather*}
for any $U,W\in H_{\bh}$, see \cite{DZ06} for details.
\par From equation \eqref{sdi13} an identity
\begin{gather*}
i\hbar
 \Bigl(\rho_{\bh}\frac{\widetilde{\Psi}^m-\breve{\Psi}^m}{\tau_m}, W\Bigr)_{H_{\bh}}
 =\Bigl(\overset{\circ}{\mathcal H}_{\bh}
 \frac{\widetilde{\Psi}^m+\breve{\Psi}^m}{2},W\Bigr)_{H_{\bh}}
 +\Bigl(\widetilde{V}_h\frac{\widetilde{\Psi}^m+\breve{\Psi}^m}{2}, W\Bigr)_{H_{\bh}}
 +(F^m, W)_{H_{\bh}}
\end{gather*}
follows, for any $m\geq 1$ and $W\in H_{\bh}$. Acting in the spirit of the proof of Proposition \ref{p12d}, i.e. choosing $W=\frac{\widetilde{\Psi}^m+\breve{\Psi}^m}{2}$, separating the imaginary part of the result and applying the pointwise equalities \eqref{p409} that are valid now on $\overline{\omega}_{\bh,\infty}$, we get
\begin{gather*}
 \frac{\hbar}{2\tau_m}
 \left(\|\sqrt{\rho_{\bh}}\Psi^m\|_{H_{\bh}}^2
 -\|\sqrt{\rho_{\bh}}\Psi^{m-1}\|_{H_{\bh}}^2\right)
 =\Ima\Bigl(F^m,\frac{\widetilde{\Psi}^m+\breve{\Psi}^m}{2}\Bigr)_{H_{\bh}}.
\end{gather*}
Multiplying both sides by $\frac{2\tau_m}{\hbar}$ and summing up the result over $m=1,\ldots,M$, we obtain
\begin{gather}
 \bigl\|\sqrt{\rho_{\bh}}\Psi^M\bigr\|_{H_{\bh}}^2
 =\bigl\|\sqrt{\rho_{\bh}}\Psi^0\bigr\|_{H_{\bh}}^2
 +\frac{2}{\hbar}\sum_{m=1}^M
 \Ima\Bigl(F^m,\frac{\widetilde{\Psi}^m+\breve{\Psi}^m}{2}\Bigr)_{H_{\bh}}.
\label{p513}
\end{gather}
The rest of the proof in fact repeats one for Proposition \ref{p12d}.
\end{proof}
\begin{corollary}
\label{c2p22d}
Let $F^m=0$ and $\Psi_{\bh}^0=0$ on $\omega_{{\bh},\infty}\backslash\,\omega_{\bh}$,
for any $m\geq 1$.
If the solution to the splitting in potential scheme \eqref{sdi11}-\eqref{sdi21} is such that $\Psi^m\in H_{\bh}$, for any $m\geq 0$, and satisfies an equation
\begin{equation}
 \Bigl.\Bigl(\overset{\circ}{\partial}_x
 \frac{\widetilde{\Psi}^m+\breve{\Psi}^m}{2}\Bigr)\Bigr|_{j=J}=
 {\mathcal S}^m\widetilde{\bPsi}^m_J\ \ \text{on}\ \ \omega_\delta,\ \
 \text{for any}\ \ m\geq 1,
\label{dtbc2d}
\end{equation}
with some operator $\mathcal{S}^m=\mathcal{S}_{\rm ref}^m$,
then the following equality holds, for any $M\geq 1$
\begin{gather}
\hbar B_{1\infty}
 \Ima\sum_{m=1}^M\left({\mathcal S}_{\rm ref}^m\bPsi_J^m,
 \overline{s}_t\Psi_{J\cdot}^m\right)_{\omega_{\delta}}\tau_m
\nonumber
 =\|\Psi^M\|_{\omega_{{\bh},\infty}\backslash\,\omega_{\bh}}^2
 :=\frac{h}{2}\,\|\Psi_{J\cdot}^M\|_{\omega_\delta}^2
 +\sum_{j=J+1}^\infty\|\Psi_{j\cdot}^M\|_{\omega_\delta}^2h\geq 0.
\label{p515}
\end{gather}
\end{corollary}
\begin{proof}
Similarly to \cite{DZ06}, equation \eqref{dtbc2d} is equivalent to the approximate TBC \eqref{sd21} provided that equation \eqref{sdi13} is valid for $j=J$ with $F|_{j=J}=0$.
Thus the solution to the splitting scheme \eqref{sdi11}-\eqref{sdi21} on the infinite mesh $\overline{\omega}_{\bf h,\infty}$ is the solution to the splitting scheme  \eqref{sd11}-\eqref{sd21} on the finite mesh $\overline{\omega}_{\bf h}$, too.
Then equality \eqref{p515} is obtained by subtracting \eqref{p413} from \eqref{p513}.
\end{proof}
\par Equality \eqref{p515} clarifies the energy sense of inequality \eqref{cs2d}
for  $\mathcal{S}^m=\mathcal{S}_{\rm ref}^m$ since
\[
 \|W\|_{H_{\bh}}^2=\|W\|_{\overline{\omega}_{\bh}}^2
 +\|W\|_{\omega_{{\bh},\infty}\backslash\,\omega_{\bh}}^2\ \
 \text{for any}\ \ W\in H_{\bh}.
\]
\par By definition, the operator ${\mathcal S}_{\rm ref}$ that has just appeared implicitly corresponds to {\it the discrete TBC}. Let us describe it explicitly.
\par Let $\overline{\omega}_{h,j_0,\infty}:=\{x_j\}_{j=j_0}^\infty$.
Then $\Delta V_{\bh}=0$ on
$\overline{\omega}_{h,J-1,\infty}\times\omega_\delta$ and clearly
\[
 \breve{\Psi}^m=\Psi^{m-1},\ \
 \Psi^m=\widetilde{\Psi}^m\ \ \text{on}\ \
 \overline{\omega}_{h,J-1,\infty}\times\overline{\omega}_\delta,\ \
 \text{for any}\ \ m\geq 1.
\]
Therefore, if  $\Psi_{\bh}^0=0$ on $\overline{\omega}_{h,J-1,\infty}\times\overline{\omega}_\delta$ and $F^m=0$ on $\overline{\omega}_{h,J,\infty}\times\overline{\omega}_\delta$ for any $m\geq 1$,
the splitting in potential scheme \eqref{sdi11}-\eqref{sdi21} is reduced
on $\overline{\omega}_{h,J-1,\infty}\times\overline{\omega}_\delta$
to the following auxiliary problem
\begin{gather}
 i\hbar\rho_\infty\overline{\partial}_t\Psi^m
 =({\mathcal H}_{0{\bf h},\infty}+V_\infty)\overline{s}_t\Psi^m
\ \ \text{on}\ \ \overline{\omega}_{h,J,\infty}\times\omega_\delta,
\label{s321}\\[1mm]
\left.\Psi^m\right|_{k=0,K}=0\ \ \text{on}\ \
 \overline{\omega}_{h,J-1,\infty},
\label{s323}\\[1mm]
 \Psi^0=0\ \ \text{on}\ \ \overline{\omega}_{h,J-1,\infty}\times\overline{\omega}_\delta,
\label{s325}
\end{gather}
for any $m\geq 1$. Here ${\mathcal H}_{0{\bf h},\infty}$ is the limiting 2D mesh Hamiltonian operator
\[
 {\mathcal H}_{0{\bf h},\infty}W
 :=-\frac{\hbar^{\,2}}{2}
 \Bigl(B_{1\infty} \widehat{\partial}_x\overline{\partial}_x W
 +B_{2\infty} \widehat{\partial}_y\overline{\partial}_y W\Bigr)
 \ \ \text{on}\ \ \omega_{{\bf h},\infty}\backslash\,\omega_{\bf h}
\]
with constant coefficients. Moreover, equation \eqref{dtbc2d} takes the simplified form
\begin{equation}
 \Bigl.\Bigl(\overset{\circ}{\partial}_x
 \overline{s}_t\Psi^m\Bigr)\Bigr|_{j=J}=
 {\mathcal S}^m\bPsi^m_J\ \ \text{on}\ \ \omega_\delta,\ \
 \text{for any}\ \ m\geq 1.
\label{s327}
\end{equation}
\par Problem \eqref{s321}-\eqref{s325} and equation \eqref{s327} are the same that correspond to the original Crank-Nicolson scheme \eqref{dse2d}-\eqref{dsi2d} and (after dividing  \eqref{s321} by $\rho_\infty$) were studied in detail in \cite{DZ06} in order to construct explicitly the discrete TBC.
We recall briefly the answer from \cite{DZ06}. To this end,
introduce the auxiliary mesh eigenvalue problem in $y$
\[
 -\widehat{\partial}_y\overline{\partial}_yE=\lambda E\ \
 \text{on}\ \ \omega_\delta,\ \ \left.E\right|_{k=0,K}=0,\ \ E\not\equiv 0.
\]
We denote by $\{E_l,\lambda_{l\delta}\}_{l=1}^{K-1}$ its eigenpairs such that the functions
$\{E_l\}_{l=1}^{K-1}$ are real-valued and form an orthonormal basis in
$\overset{\circ}{H}(\overline{\omega}_\delta)$; here $\lambda_{l\delta}>0$
for all $l$. Clearly, for any $U\in\overset{\circ}{H}(\overline{\omega}_\delta)$,
the following expansion holds
\[
 U=\mathcal{F}^{-1}U^{(\cdot)}:=\sum_{l=1}^{K-1}U^{(l)}E_l,\ \ \text{where}\ \
 U^{(l)}=(\mathcal{F}U)^{(l)}:=(U,E_l)_{\omega_\delta}\ \
 \text{for}\ \ 1\leq l\leq K-1.
\]
These formulas define the direct $\mathcal{F}$ and inverse $\mathcal{F}^{-1}$ transforms from the collection of values $\{U_k\}_{k=1}^{K-1}$ to the collection of its Fourier coefficients $\{U^{(l)}\}_{l=1}^{K-1}$ and back.
\par In the case of the uniform mesh $\overline{\omega}_\delta$,
i.e. $\delta_k=\delta$ for any $1\leq k\leq K$,
the eigenpairs are represented explicitly by the well-known formulas
\[
 (E_l)_k:=\sqrt{\frac{2}{Y}}\,\sin\frac{\pi ly_k}{Y},\ \ 0\leq k\leq K,\ \
 \lambda_{l\delta}
 :=\left(\frac{2}{\delta}\,\sin\frac{\pi\delta l}{2Y}\right)^2,\ \
 \text{for}\ \ 1\leq l\leq K-1,
\]
and the transforms $\mathcal{F}$ and $\mathcal{F}^{-1}$ can be
effectively implemented by applying the discrete fast Fourier
transform (FFT) with respect to sines.
\par Let the mesh in time $\overline{\omega}^{\,\tau}$ be uniform. Recall
the discrete convolution
\[
 (R*Q)^m:=\sum_{q=0}^mR^qQ^{m-q}\ \ \text{for any}\ \ m\geq 0
\]
of the sequences
$R,Q$: $\overline{\omega}^{\,\tau}\to\mathbb{C}$.
The operator ${\mathcal S}_{\rm ref}$ is given by a discrete convolution in $t$
\begin{equation}
 {\mathcal S}_{\rm ref}^m\bPhi^m
 =\frac{1}{2h}\,\mathcal{F}^{-1}\Bigl(R_l*(\mathcal{F}\Phi)^{(l)}\Bigr)^m\ \
 \text{for any}\ \ m\geq 1
\label{S2d}
\end{equation}
also involving the above transforms $\mathcal{F}$ and $\mathcal{F}^{-1}$ in $y$.
Expressions for the kernel sequences $R_l$, $1\leq l\leq K-1$, can be found in \cite{DZ06}, and we do not reproduce them here (see also \cite{EA01,DZZ09} for practically more convenient recurrence relations).
\begin{proposition}
\label{p3d2}
The operator ${\mathcal S}_{\rm ref}$ satisfies inequality
\eqref{cs2d}, see \cite{DZ06}.
\par Thus for the solution of the splitting in potential scheme  \eqref{sd11}-\eqref{sd21}
with the discrete TBC (i.e. with ${\mathcal S}={\mathcal S}_{\rm ref}$),
the stability bound \eqref{sb2d} holds.
\end{proposition}
\par If the matrix ${\mathbf B}$
is diagonal and $y$-independent together with $\rho$
as well as the mesh $\overline{\omega}_\delta$ is uniform,
the splitting in potential scheme \eqref{sd11}-\eqref{sd21} with the discrete TBC can be effectively implemented. Indeed, then
(omitting $F$ for simplicity)
we can apply $\mathcal{F}$ to the main equation \eqref{sd13} and the discrete TBC \eqref{sd19} with ${\mathcal S}={\mathcal S}_{\rm ref}$,
take into account the boundary condition $\widetilde{\Psi}^m|_{\Gamma_{\bh}}=0$ \eqref{sd19} and formula \eqref{S2d} and obtain a collection of 1D disjoint problems:
\begin{gather}
 i\hbar\rho_h\frac{\widetilde{\Psi}^{m(l)}-\breve{\Psi}^{m(l)}}{\tau}
 =-\frac{\hbar^2}{2}\widehat{\partial}_x\Bigl(B_{11h}\overline{\partial}_x
 \frac{\widetilde{\Psi}^{m(l)}+\breve{\Psi}^{m(l)}}{2}\Bigr)
 +\widetilde{V}_{l}\frac{\widetilde{\Psi}^{m(l)}+\breve{\Psi}^{m(l)}}{2}
 \ \ \text{on}\ \ \omega_h,
\label{sd513}\\[1mm]
 \bigl.\widetilde{\Psi}^{m(l)}\bigr|_{j=0}=0,
 \label{sd517}
\end{gather}
\begin{gather}
 \left.
 \left\{
 \frac{\hbar^{\,2}}{2}B_{1\infty}\overline{\partial}_x
 \frac{\widetilde{\Psi}^{m(l)}+\breve{\Psi}^{m(l)}}{2}
 -\frac{h}{2}
 \Bigl[i\hbar\rho_{\infty}\frac{\widetilde{\Psi}^{m(l)}-\breve{\Psi}^{m(l)}}{\tau}
 -\widetilde{V}_{l}\frac{\widetilde{\Psi}^{m(l)}+\breve{\Psi}^{m(l)}}{2}\Bigr]
 \right\}\right|_{j=J}
\nonumber\\[1mm]
 =\frac{\hbar^{\,2}}{2}B_{1\infty}
 \frac{1}{2h}\Bigl(R_l*\widetilde{\bPsi}^{(l)}_J\Bigr)^m,
\label{sd519}
\end{gather}
with $1\leq l\leq K-1$, for any $m\geq 1$, where
\[
 \widetilde{V}_{l}:=\frac{\hbar^{\,2}}{2}B_{2\infty}\lambda_{l\delta}
 +\widetilde{V},\ \
 \widetilde{\bPsi}^{m(l)}_J
 =\Bigl\{\widetilde{\Psi}^{1(l)}_J,\dots,\widetilde{\Psi}^{m(l)}_J\Bigr\}.
\]
\par Given $\Psi^{m-1}$, the direct algorithm for computing $\Psi^m$ comprises five steps.
\begin{enumerate}
\item $\breve{\Psi}^m$ is computed on $(\omega_h\cup{x_J})\times\omega_{\delta}$ according to \eqref{sd23}.
\item $\Bigl\{(\breve{\Psi}^m_{j\cdot})^{(l)}\Bigr\}_{l=1}^{K-1}$ is computed by applying $\mathcal{F}$ for any $1\leq j\leq J$.
\item $\Bigl\{\widetilde{\Psi}^{m(l)}_j\Bigr\}_{j=1}^J$ is computed by solving problem \eqref{sd513}-\eqref{sd519} for any $1\leq l\leq K-1$.
\item $\Bigl\{\widetilde{\Psi}^m_{jk}\Bigr\}_{k=1}^{K-1}$ is computed by applying $\mathcal{F}^{-1}$ for any $1\leq j\leq J$.
\item $\Psi^m$ is computed on $(\omega_h\cup{x_J})\times\omega_{\delta}$ according to \eqref{sd23}.
\end{enumerate}
\par Steps 1 and 5 require $O(JK)$ arithmetic operations, steps 2 and 4 require $O(JK\log_2K)$ operations using FFT provided that $K=2^p$ with the integer $p$, and step 3 requires $O((J+m)K)$ operations. The total amount of operations is $O((J\log_2K+m)K)$ for computing the solution on the time level $m$ and $O((J\log_2K+M)KM)$ for computing the solution on $M$ time levels $m=1,\ldots,M$.
\par Notice that the algorithm essentially enlarges possibilities of the corresponding one in \cite{ZZ11,IZ11} and also is highly parallelizable.
\section{Numerical experiments}
\label{sectnum}
We have implemented the above algorithm.
For numerical experiments, similarly to \cite{IZ10}, we take $\rho(x,y)\equiv 1$, $\mathcal{H}_0=-\Delta$, $\hbar=1$ and a simple rectangular potential-barrier
\[
 V(x,y)=
 \begin{cases}
Q &\text{for}\ \ (x,y)\in (a,b)\times (c,d)\\[1mm]
0 &\text{otherwise}
\end{cases},\ \ Q>0.
\]
On the other hand, from the numerical point of view, this barrier is not so simple since it is discontinuous and thus the corresponding exact solution is not smooth. Below we take the fixed $(a,b)=(1.6,1.7)$ and $(c,d)$ of three different lengths in Examples A, B and C.
\par Let the initial function be the Gaussian wave package
\[
 \psi^0(x,y)=
 \psi_G(x, y) \equiv
 \exp\left\{ i\sqrt{2}k(x-x^{(0)})-\frac{(x-x^{(0)})^2+(y-y^{(0)})^2}
 {4\alpha}\right\}\ \ \text{on}\ \ \mathbb{R}^2.
\]
We choose the parameters $k=30$ and $\alpha=\frac{1}{120}$.
\par We solve the initial boundary value problem in the infinite strip $\mathbb{R}\times (0,Y)$, choose the computational domain $\bar{\Omega}_X\times[0,T]$ such that $(a,b)\times (c,d)\subset \Omega_X$ and $\psi_G$ is small enough outside $\bar{\Omega}_X$. Namely, below $X=3$ and $Y=2.8$ together with
$(x^{(0)},y^{(0)})=(1,\frac{Y}{2})$ as well as $T=t_M=0.027$.
We use the uniform meshes in $x$, $y$ and $t$
with steps respectively $h=\frac{X}{J}$, $\delta=\frac{Y}{K}$ and $\tau=\frac{T}{M}$.
\par We accordingly modify the splitting in potential scheme \eqref{sd11}-\eqref{sd21} replacing $\omega_h\cup{x_J}$ by $\overline{\omega}_h$ in \eqref{sd11} and \eqref{sd15}, the boundary conditions \eqref{sd17} by
\[
 \breve{\Psi}^m|_{k=0,K}=0,\ \ \widetilde{\Psi}^m|_{k=0,K}=0,\ \
 \Psi^m|_{k=0,K}=0\ \ \text{on}\ \ \overline{\omega}_h
\]
together with the left discrete TBC at $j=0$ similar to the right one \eqref{sd19} for ${\mathcal S}={\mathcal S}_{\rm ref}$.
On Figure \ref{SSP:EX22a:B:Initial} the modulus and the real part of $\psi_G$ are shown on the computational domain. The normalized barrier (with $Q=1$) is situated there as well, for $(c,d)=(0.7,2.1)$ (see Example B below).
\begin{figure}[ht]\centerline{
    \includegraphics[width=0.5\linewidth]{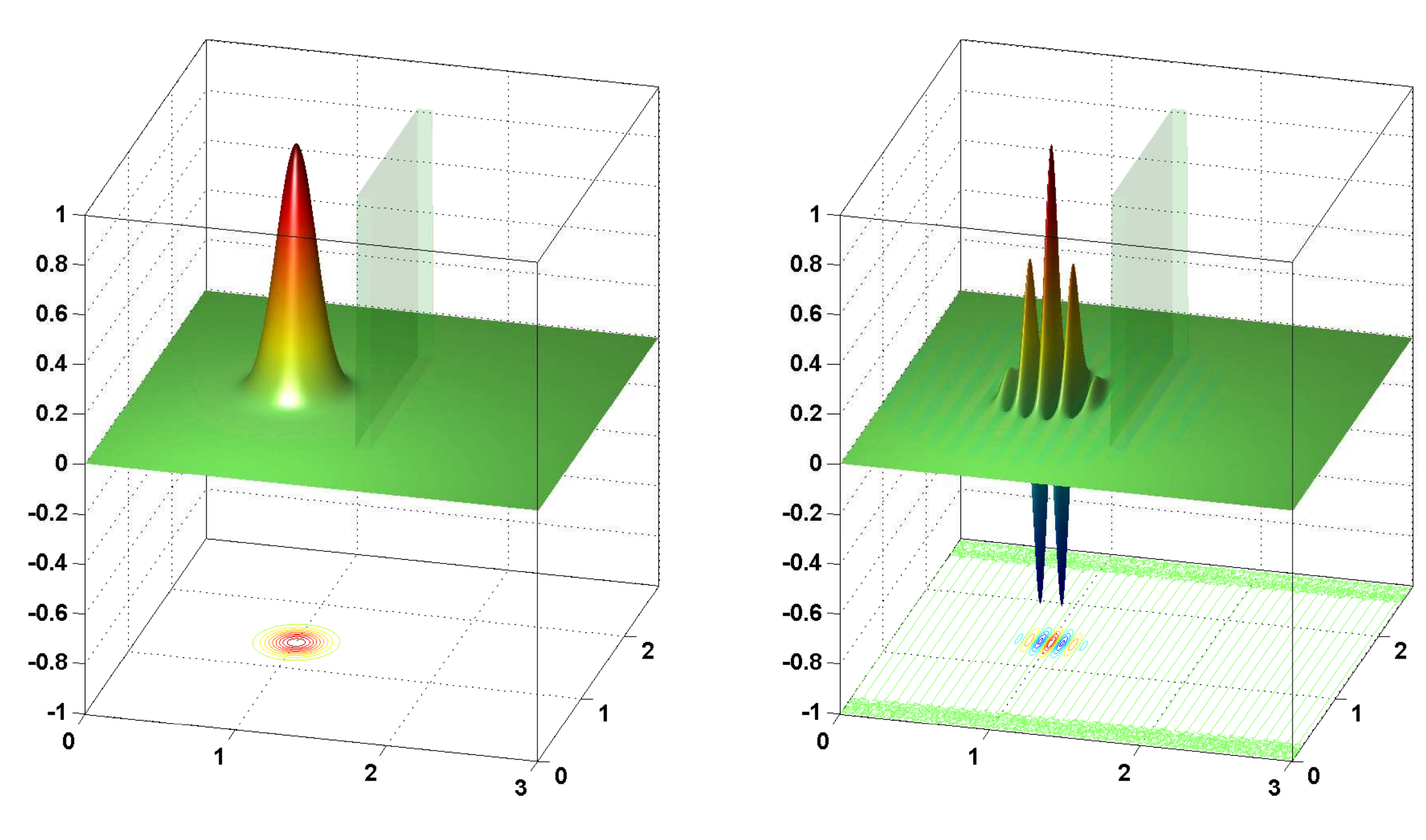}}
\caption{\small{The modulus (left) and the real part (right) of the initial function $\psi_G$} together with the normalized barrier from Example B
\label{SSP:EX22a:B:Initial}}
\end{figure}
\smallskip\par \textbf{Example A}. We first take $(c,d)=(0,Y)$ and $Q=1500$ (previously the same example was in fact treated in \cite{IZ10} by a method from \cite{ZZ11,IZ11}). In this case, the wave package is divided into two similar reflected and transmitted parts moving in opposite directions with respect to the barrier. A little bit surprisingly, the solution is almost the same as in the case $(c,d)=(\frac{Y}{4},\frac{3Y}{4})$ in Example B (that is why the corresponding graphs of the solution are given below).
Namely, for the fine mesh with
$(J,K,M)=(9600,512,4800)$,
norms of differences between the pseudo-exact solution for $(c,d)=(0,Y)$ computed by the Crank-Nicolson scheme and one for $(c,d)=(\frac{Y}{4},\frac{3Y}{4})$ computed by the splitting in potential scheme are
\[
 E_{C}\approx 1.81\cdot10^{-3},\ \ E_{L^2}\approx 5.02\cdot10^{-4},
\]
i.e., they are actually small.
In this section, we exploit the splitting method with the simplest choice $\widetilde{V}=0$ unless the contrary is explicitly stated.
Hereafter $E_{C}$ and $E_{L^2}$ denote differences/errors in the mesh norms that are uniform in time as well as $C$ (i.e. uniform) and $L^2$ in space.
\par Notice also that the norms of differences between the pseudo-exact solutions on the fine mesh and on the mesh with $(J,K,M)=(1200,64,600)$
are
\[
 E_{C}\approx2.70\cdot10^{-2},\ \ E_{L^2}\approx1.65\cdot10^{-2},
\]
for the Crank-Nicolson scheme, and
\[
 E_{C}\approx2.54\cdot10^{-2},\ \ E_{L^2}\approx1.60\cdot10^{-2},
\]
for the splitting scheme (see also Figure \ref{SSP:EX22a:B:Errors} below for more detail), i.e., they are small enough and close to each other.
\smallskip\par \textbf{Example B}. Next we consider the barrier with $(c,d)=(\frac{Y}{4},\frac{3Y}{4})=(0.7,2.1)$  (that is one-half in length of the first one), for three values of the barrier height $Q$, in order to get qualitatively varying behavior of solutions (compare with \cite{IZ10}).
\par First, once again we take the barrier height $Q=1500$.
We exploit the mesh with $(J,K,M)=(1200,64,600)$
so that
$h=2.5\cdot 10^{-3}$, $\delta=4.375\cdot10^{-2}$ and $\tau=4.5\cdot 10^{-5}$.
Though $\frac{J}{K}\approx 19$ is large, we qualify that such choice is reasonable.
In particular, for comparison we exploit the above mentioned pseudo-exact solution on the fine mesh with $(J,K,M)=(9600,512,4800)$
that all three are 8 times larger and imply the
steps $h=3.125\cdot 10^{-4}$, $\delta\approx 5.469\cdot10^{-3}$ and $\tau=5.625\cdot 10^{-6}$.
Figure \ref{SSP:EX22a:B:Errors} demonstrates the behavior of the absolute and relative errors in $C$ and $L^2$ space mesh norms in dependence with time.
\begin{figure}[ht]\centerline{
    \includegraphics[width=0.5\linewidth]{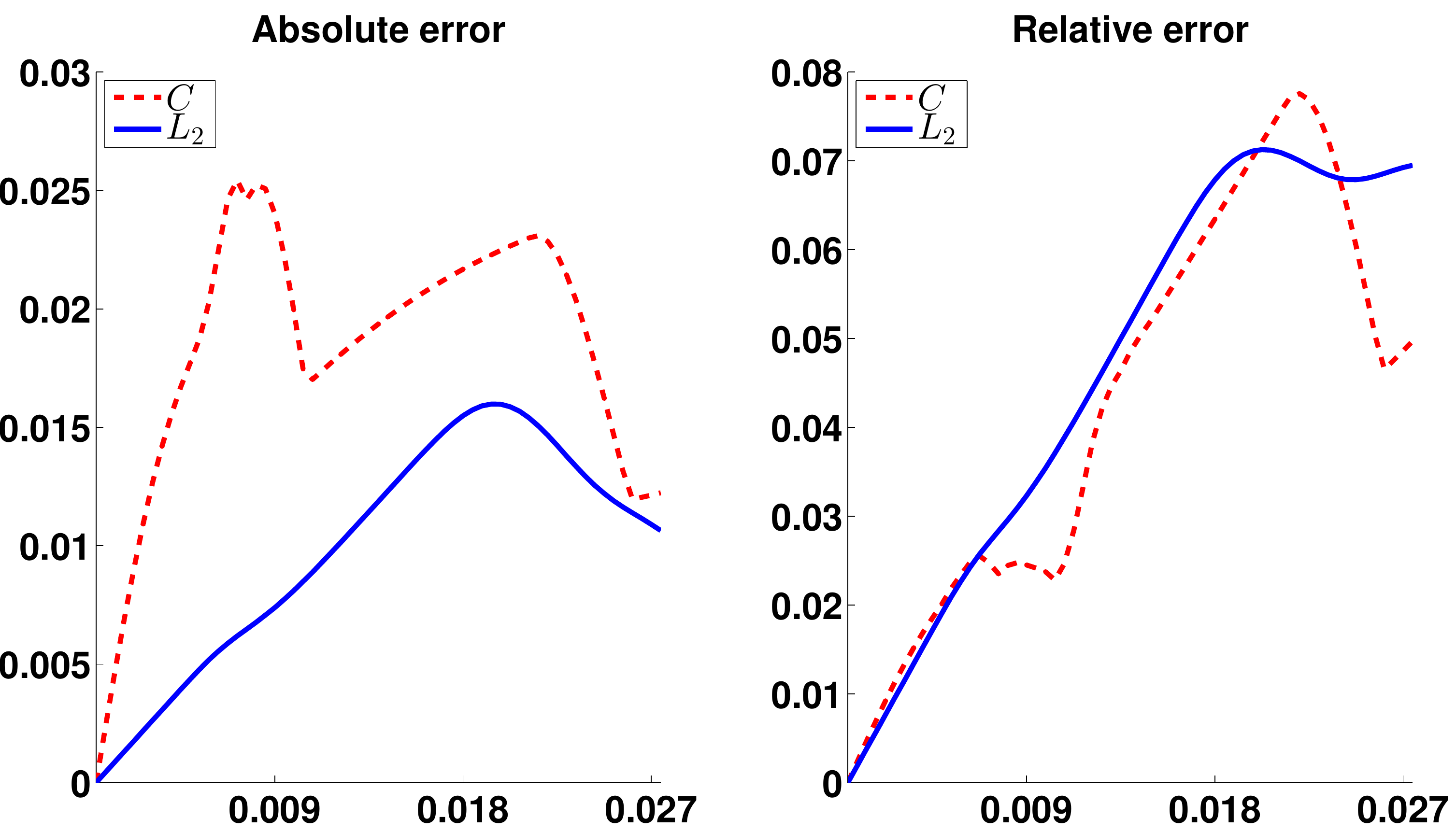}}
\caption{\small{Example B ($Q=1500$). The absolute and relative errors in $C$ and $L^2$ norms in dependence with time for the numerical solution for $(J,K,M)=(1200,64,600)$}
\label{SSP:EX22a:B:Errors}}
\end{figure}
\par The modulus and the real part of the numerical solution $\Psi^{m}$ are given on Figure \ref{SSP:EX22a:B:Solution}, for the time moments $t_m=m\tau$, $m=180, 300, 420$ and $600$. In this case, the wave package is divided into two rather similar reflected and transmitted parts moving in opposite directions with respect to the barrier.
\begin{figure}[ht]
\begin{multicols}{2}
    \includegraphics[scale=0.25]{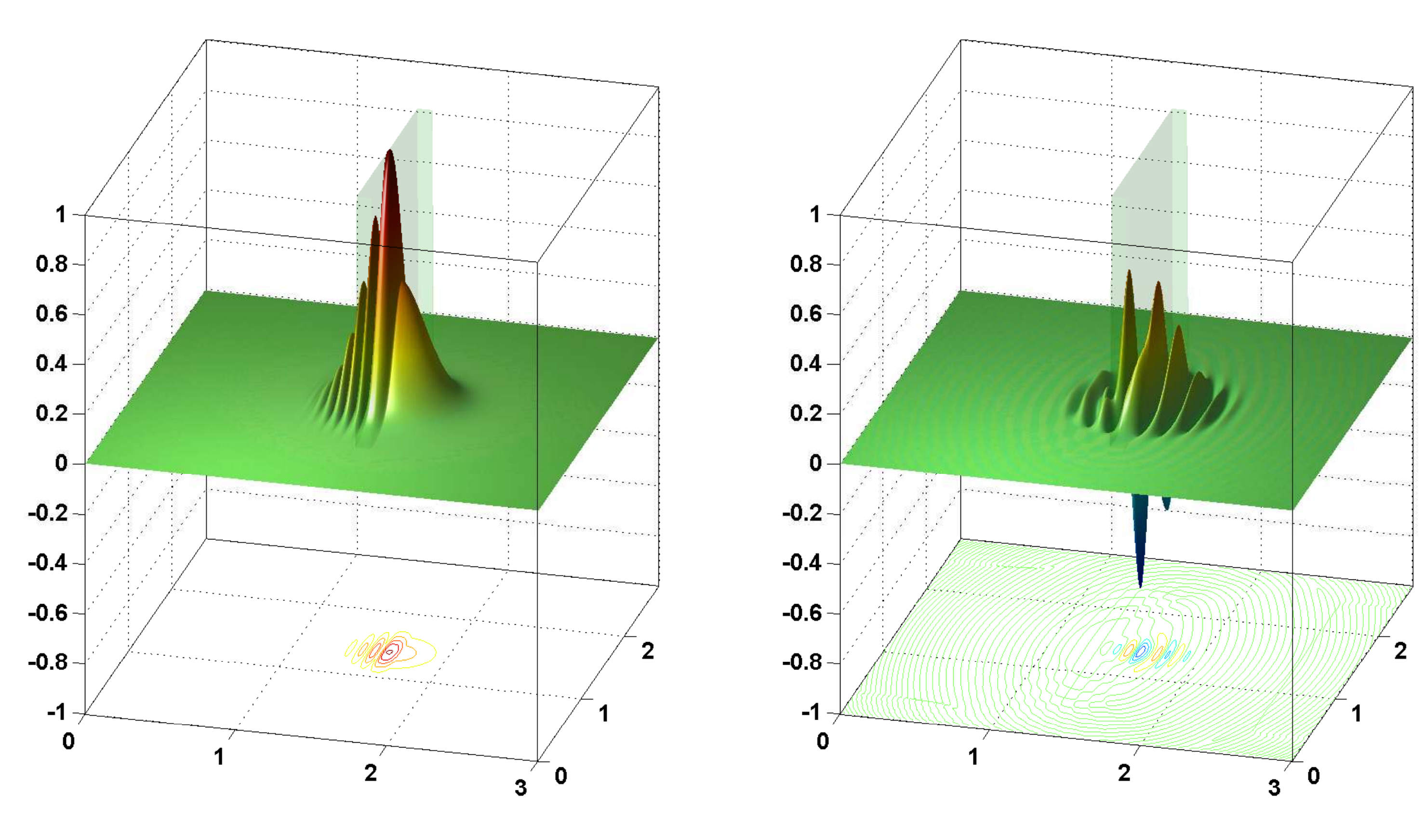}
    \vspace{0cm}\\
    \centerline{\small{$m=180$}}\\
    \includegraphics[scale=0.25]{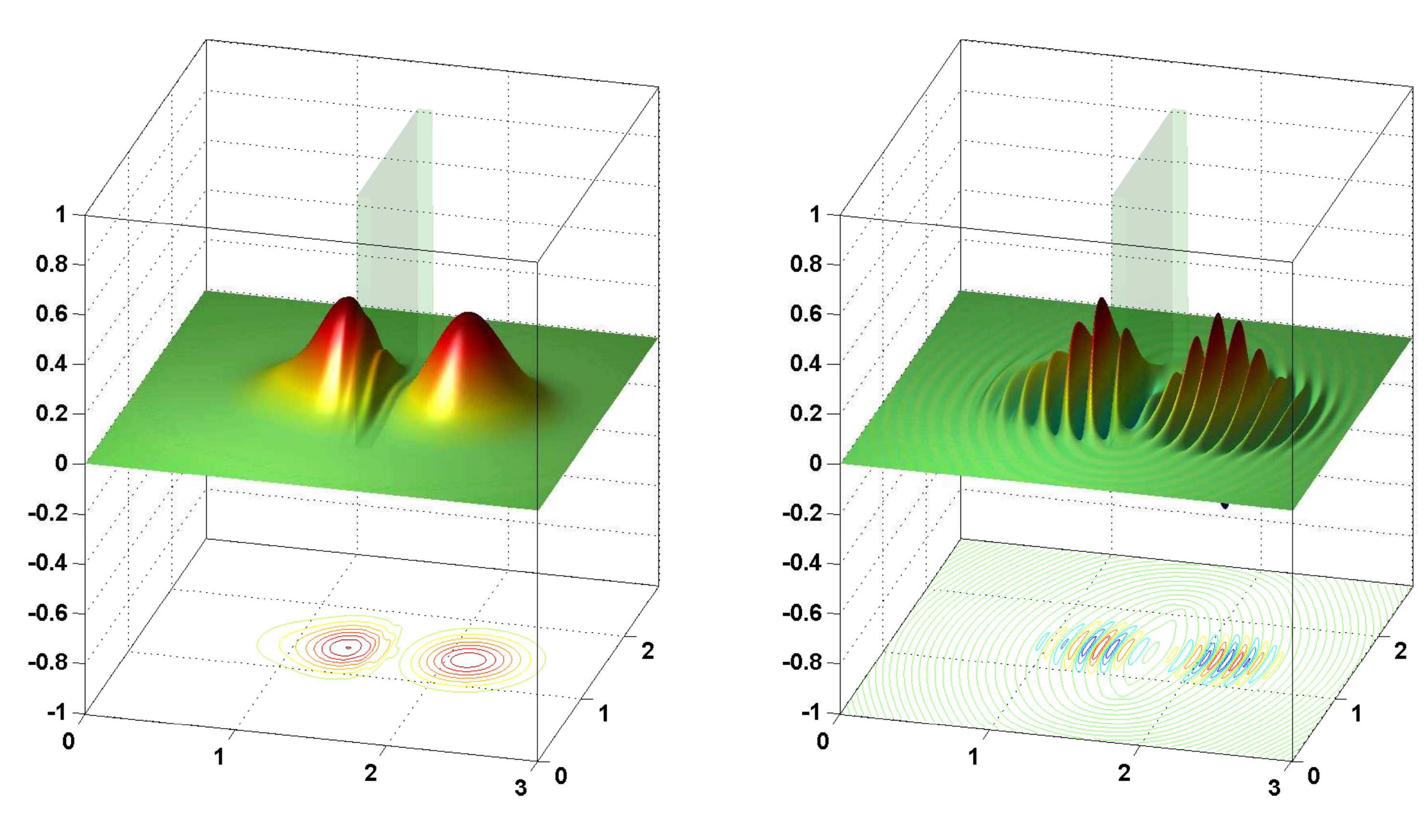}
    \vspace{0cm}\\
    \centerline{\small{$m=300$}}\\
\end{multicols}
\begin{multicols}{2}
    \includegraphics[scale=0.25]{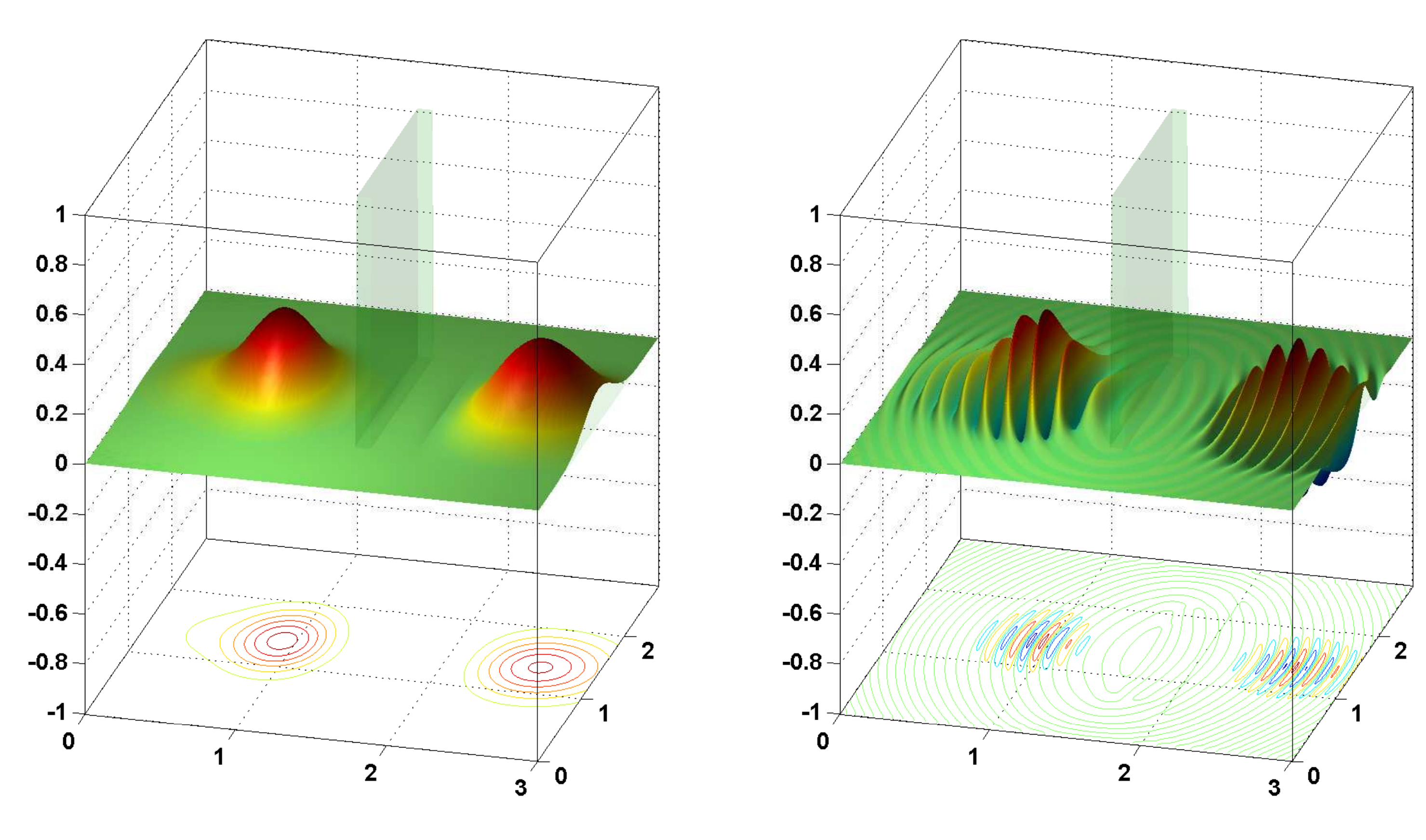}
    \vspace{0cm}\\
    \centerline{\small{$m=420$}}\\
    \includegraphics[scale=0.25]{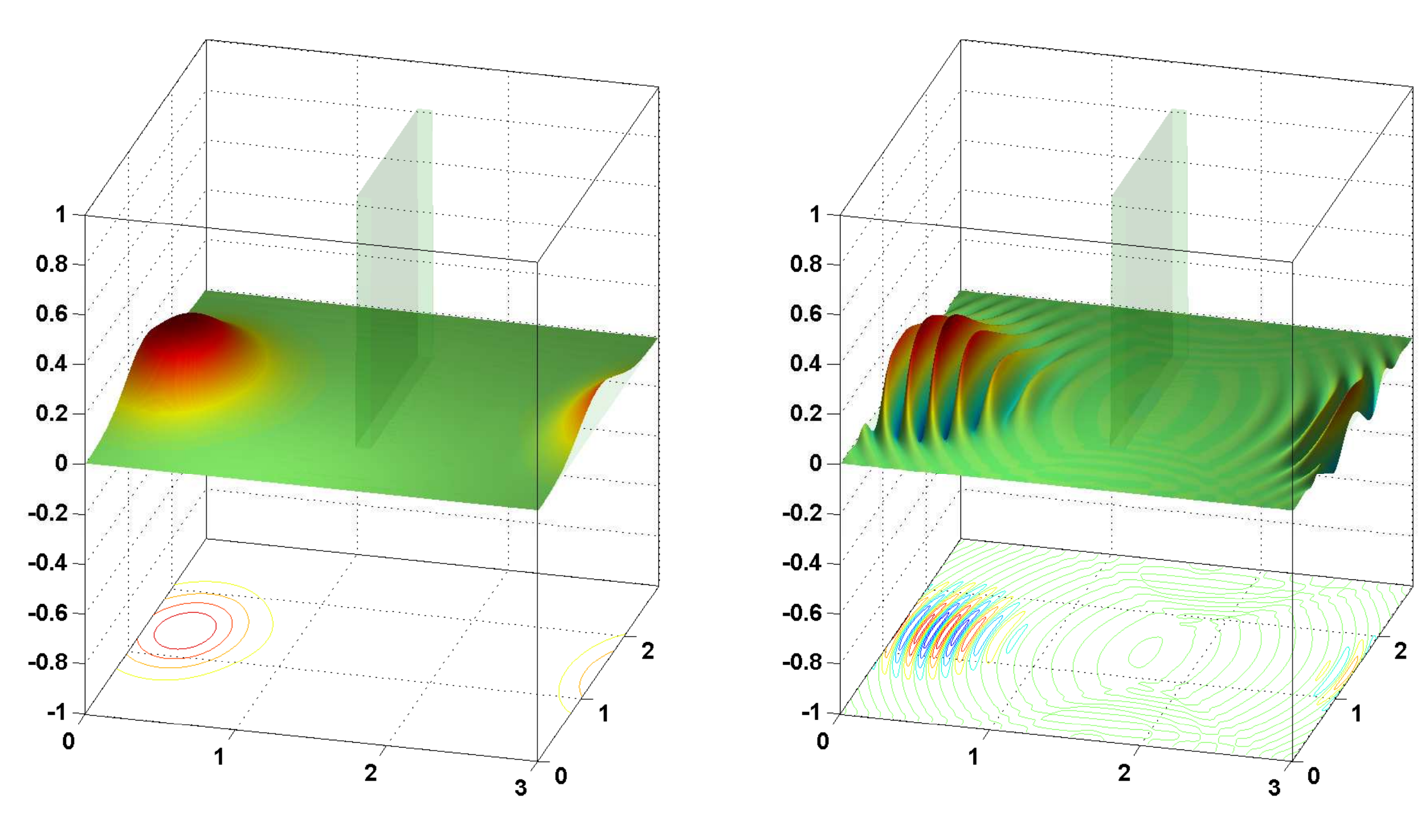}
    \vspace{0cm}\\
    \centerline{\small{$m=600$}}\\
    \end{multicols}
\caption{\small{Example B ($Q=1500$). The modulus and the real part of the numerical solution $\Psi^m$
for $(J,K,M)=(1200,64,600)$}}
\label{SSP:EX22a:B:Solution}
\end{figure}
%
%
\par We continue to study the error behavior in more detail in
Table \ref{SSP:tab:EX22a:M}. It contains errors of the solutions to the splitting method for increasing $J$, $K$ and $M$ respectively (for sufficiently large values of two other numbers).
The associated ratios $R_C$ and $R_{L^2}$ of the sequential errors are also put there. They are rather close to 4 excluding the last rows that means almost the second order of convergence with respect to each of $J$, $K$ and $M$, both in $C$ and $L^2$ mesh norms. The deterioration of $R_C$ and $R_{L^2}$ in the last rows is explained by more essential influence of the errors due to the chosen discretization in other directions.
\par The last columns in the tables contain also the respective ratios of runtimes.
One can see that they all are close to 2 when any of $J$, $K$ or $M$ increases twice.
\begin{table}
    \centering{
    \begin {tabular}{r<{\pgfplotstableresetcolortbloverhangright }@{}l<{\pgfplotstableresetcolortbloverhangleft }r<{\pgfplotstableresetcolortbloverhangright }@{}l<{\pgfplotstableresetcolortbloverhangleft }r<{\pgfplotstableresetcolortbloverhangright }@{}l<{\pgfplotstableresetcolortbloverhangleft }r<{\pgfplotstableresetcolortbloverhangright }@{}l<{\pgfplotstableresetcolortbloverhangleft }r<{\pgfplotstableresetcolortbloverhangright }@{}l<{\pgfplotstableresetcolortbloverhangleft }r<{\pgfplotstableresetcolortbloverhangright }@{}l<{\pgfplotstableresetcolortbloverhangleft }}%
\toprule \multicolumn {2}{c}{$J$}&\multicolumn {2}{c}{$E_C$}&\multicolumn {2}{c}{$R_{C}$}&\multicolumn {2}{c}{$E_{L^2}$}&\multicolumn {2}{c}{$R_{L^2}$}&\multicolumn {2}{c}{$R_{\rm time}$}\\\midrule %
$300$&$$&$0$&$.22$&--&&$0$&$.12$&--&&--&\\%
$600$&$$&$5$&$.62\cdot 10^{-2}$&$3$&$.99$&$3$&$.10\cdot 10^{-2}$&$3$&$.94$&$1$&$.33$\\%
$1\,200$&$$&$1$&$.42\cdot 10^{-2}$&$3$&$.95$&$7$&$.91\cdot 10^{-3}$&$3$&$.92$&$1$&$.47$\\%
$2\,400$&$$&$3$&$.77\cdot 10^{-3}$&$3$&$.77$&$2$&$.16\cdot 10^{-3}$&$3$&$.66$&$1$&$.63$\\%
$4\,800$&$$&$1$&$.20\cdot 10^{-3}$&$3$&$.15$&$7$&$.65\cdot 10^{-4}$&$2$&$.83$&$1$&$.81$\\\bottomrule %
\end {tabular}%
\\[2mm]
    \begin {tabular}{r<{\pgfplotstableresetcolortbloverhangright }@{}l<{\pgfplotstableresetcolortbloverhangleft }r<{\pgfplotstableresetcolortbloverhangright }@{}l<{\pgfplotstableresetcolortbloverhangleft }r<{\pgfplotstableresetcolortbloverhangright }@{}l<{\pgfplotstableresetcolortbloverhangleft }r<{\pgfplotstableresetcolortbloverhangright }@{}l<{\pgfplotstableresetcolortbloverhangleft }r<{\pgfplotstableresetcolortbloverhangright }@{}l<{\pgfplotstableresetcolortbloverhangleft }r<{\pgfplotstableresetcolortbloverhangright }@{}l<{\pgfplotstableresetcolortbloverhangleft }}%
\toprule \multicolumn {2}{c}{$K$}&\multicolumn {2}{c}{$E_{C}$}&\multicolumn {2}{c}{$R_{C}$}&\multicolumn {2}{c}{$E_{L^2}$}&\multicolumn {2}{c}{$R_{L^2}$}&\multicolumn {2}{c}{$R_{\rm time}$}\\\midrule %
$16$&$$&$0$&$.15$&--&&$6$&$.62\cdot 10^{-2}$&--&&--&\\%
$32$&$$&$3$&$.43\cdot 10^{-2}$&$4$&$.36$&$2$&$.34\cdot 10^{-2}$&$2$&$.84$&$2$&$.13$\\%
$64$&$$&$8$&$.65\cdot 10^{-3}$&$3$&$.96$&$6$&$.58\cdot 10^{-3}$&$3$&$.55$&$1$&$.91$\\%
$128$&$$&$2$&$.29\cdot 10^{-3}$&$3$&$.79$&$1$&$.81\cdot 10^{-3}$&$3$&$.63$&$1$&$.86$\\%
$256$&$$&$1$&$.20\cdot 10^{-3}$&$1$&$.91$&$7$&$.65\cdot 10^{-4}$&$2$&$.37$&$1$&$.87$\\\bottomrule %
\end {tabular}%
\\[2mm]
    \begin {tabular}{r<{\pgfplotstableresetcolortbloverhangright }@{}l<{\pgfplotstableresetcolortbloverhangleft }r<{\pgfplotstableresetcolortbloverhangright }@{}l<{\pgfplotstableresetcolortbloverhangleft }r<{\pgfplotstableresetcolortbloverhangright }@{}l<{\pgfplotstableresetcolortbloverhangleft }r<{\pgfplotstableresetcolortbloverhangright }@{}l<{\pgfplotstableresetcolortbloverhangleft }r<{\pgfplotstableresetcolortbloverhangright }@{}l<{\pgfplotstableresetcolortbloverhangleft }r<{\pgfplotstableresetcolortbloverhangright }@{}l<{\pgfplotstableresetcolortbloverhangleft }}%
\toprule \multicolumn {2}{c}{$M$}&\multicolumn {2}{c}{$E_{C}$}&\multicolumn {2}{c}{$R_{C}$}&\multicolumn {2}{c}{$E_{L^2}$}&\multicolumn {2}{c}{$R_{L^2}$}&\multicolumn {2}{c}{$R_{\rm time}$}\\\midrule %
$150$&$$&$0$&$.17$&--&&$9$&$.10\cdot 10^{-2}$&--&&--&\\%
$300$&$$&$4$&$.39\cdot 10^{-2}$&$3$&$.91$&$2$&$.37\cdot 10^{-2}$&$3$&$.84$&$2$&$.1$\\%
$600$&$$&$1$&$.12\cdot 10^{-2}$&$3$&$.9$&$6$&$.20\cdot 10^{-3}$&$3$&$.83$&$2$&$.02$\\%
$1\,200$&$$&$3$&$.16\cdot 10^{-3}$&$3$&$.56$&$1$&$.83\cdot 10^{-3}$&$3$&$.39$&$2$&$.01$\\%
$2\,400$&$$&$1$&$.20\cdot 10^{-3}$&$2$&$.64$&$7$&$.65\cdot 10^{-4}$&$2$&$.39$&$2$&$.09$\\\bottomrule %
\end {tabular}%

\caption{Example B ($Q=1500$). Errors, ratios of errors and ratios of runtimes in dependence with
$J$ (for $K=256$ and $M=2400$),
$K$ (for $J=4800$ and $M=2400$) or
$M$ (for $J=4800$ and $K=256$)}
\label{SSP:tab:EX22a:M}}
\end{table}
\par In Table \ref{SSP:tab:EX22a:comparison} we put $C$ and $L^2$ errors for some selected values of $J$, $K$ and $M$.
They all decrease monotonically as $J$, $K$ or $M$ increase.
We also compare there the numerical solutions of the splitting method with $\widetilde{V}=0$ and $\widetilde{V}(x)=Q\chi(x)$, where $\chi(x)$ is the characteristic function of the interval $(a,b)$; two last columns of the table contain percentages
\[
 P_{C}:=\Bigl(\frac{\left.E_{C}\right|_{\widetilde{V}=0}}
            {\left.E_{C}\right|_{\widetilde{V}=Q\chi}}-1\Bigr)\cdot 100\%,\ \
 P_{L^2}:=\Bigl(\frac{\left.E_{L^2}\right|_{\widetilde{V}=0}}
            {\left.E_{L^2}\right|_{\widetilde{V}=Q\chi}}-1\Bigr)\cdot 100\%.
\]
One can see that the second choice $\widetilde{V}=Q\chi$ also works but the first one $\widetilde{V}=0$ mostly leads to better results.
\begin{table}\centering{
    \begin {tabular}{r<{\pgfplotstableresetcolortbloverhangright }@{}l<{\pgfplotstableresetcolortbloverhangleft }r<{\pgfplotstableresetcolortbloverhangright }@{}l<{\pgfplotstableresetcolortbloverhangleft }r<{\pgfplotstableresetcolortbloverhangright }@{}l<{\pgfplotstableresetcolortbloverhangleft }r<{\pgfplotstableresetcolortbloverhangright }@{}l<{\pgfplotstableresetcolortbloverhangleft }r<{\pgfplotstableresetcolortbloverhangright }@{}l<{\pgfplotstableresetcolortbloverhangleft }r<{\pgfplotstableresetcolortbloverhangright }@{}l<{\pgfplotstableresetcolortbloverhangleft }r<{\pgfplotstableresetcolortbloverhangright }@{}l<{\pgfplotstableresetcolortbloverhangleft }}%
\toprule \multicolumn {2}{c}{$J$}&\multicolumn {2}{c}{$K$}&\multicolumn {2}{c}{$M$}&\multicolumn {2}{c}{$E_{C}$}&\multicolumn {2}{c}{$E_{L^2}$}&\multicolumn {2}{c}{$P_{C}$}&\multicolumn {2}{c}{$P_{L^2}$}\\\midrule %
$1\,200$&$$&$64$&$$&$600$&$$&$2$&$.54\cdot 10^{-2}$&$1$&$.60\cdot 10^{-2}$&$-5$&$.65$&$-3$&$.08$\\%
$1\,200$&$$&$128$&$$&$600$&$$&$2$&$.42\cdot 10^{-2}$&$1$&$.37\cdot 10^{-2}$&$-6$&$.6$&$-3$&$.78$\\%
$1\,200$&$$&$64$&$$&$1\,200$&$$&$1$&$.83\cdot 10^{-2}$&$1$&$.21\cdot 10^{-2}$&$0$&$.36$&$-0$&$.91$\\%
$1\,200$&$$&$128$&$$&$1\,200$&$$&$1$&$.63\cdot 10^{-2}$&$9$&$.40\cdot 10^{-3}$&$-2$&$.49$&$-1$&$.33$\\%
$2\,400$&$$&$64$&$$&$600$&$$&$1$&$.64\cdot 10^{-2}$&$1$&$.09\cdot 10^{-2}$&$-5$&$.48$&$-3$&$.99$\\%
$2\,400$&$$&$128$&$$&$600$&$$&$1$&$.38\cdot 10^{-2}$&$8$&$.05\cdot 10^{-3}$&$-11$&$.39$&$-6$&$.27$\\%
$2\,400$&$$&$64$&$$&$1\,200$&$$&$1$&$.03\cdot 10^{-2}$&$7$&$.73\cdot 10^{-3}$&$1$&$.01$&$-0$&$.96$\\%
$2\,400$&$$&$128$&$$&$1\,200$&$$&$6$&$.00\cdot 10^{-3}$&$3$&$.82\cdot 10^{-3}$&$-6$&$.3$&$-3$&$.02$\\\bottomrule %
\end {tabular}%

\caption{Example B ($Q=1500$). Errors of the numerical solutions for $\widetilde{V}=0$ and percentages of their changes when taking $\widetilde{V}=Q\chi$}
\label{SSP:tab:EX22a:comparison}}
\end{table}
\par In addition, for the fine mesh with $(J,K,M)=(9600,512,4800)$
the norms of differences between the solutions for these two different $\widetilde{V}$ are
\[
 E_{C}\approx3.32\cdot 10^{-5},\ \ E_{L^2}\approx1.04\cdot 10^{-5},
\]
i.e., they are very small.
\par Second, we take a less barrier height $Q=1000$. This situation is simpler from the numerical point of view. The numerical results are demonstrated on Figure \ref{SSP:EX22a:B:Q=1000:Solution} for the same time moments and the mesh. Now the wave package goes through the barrier with an essentially less reflection.
\begin{figure}[ht]
\begin{multicols}{2}
    \includegraphics[scale=0.25]{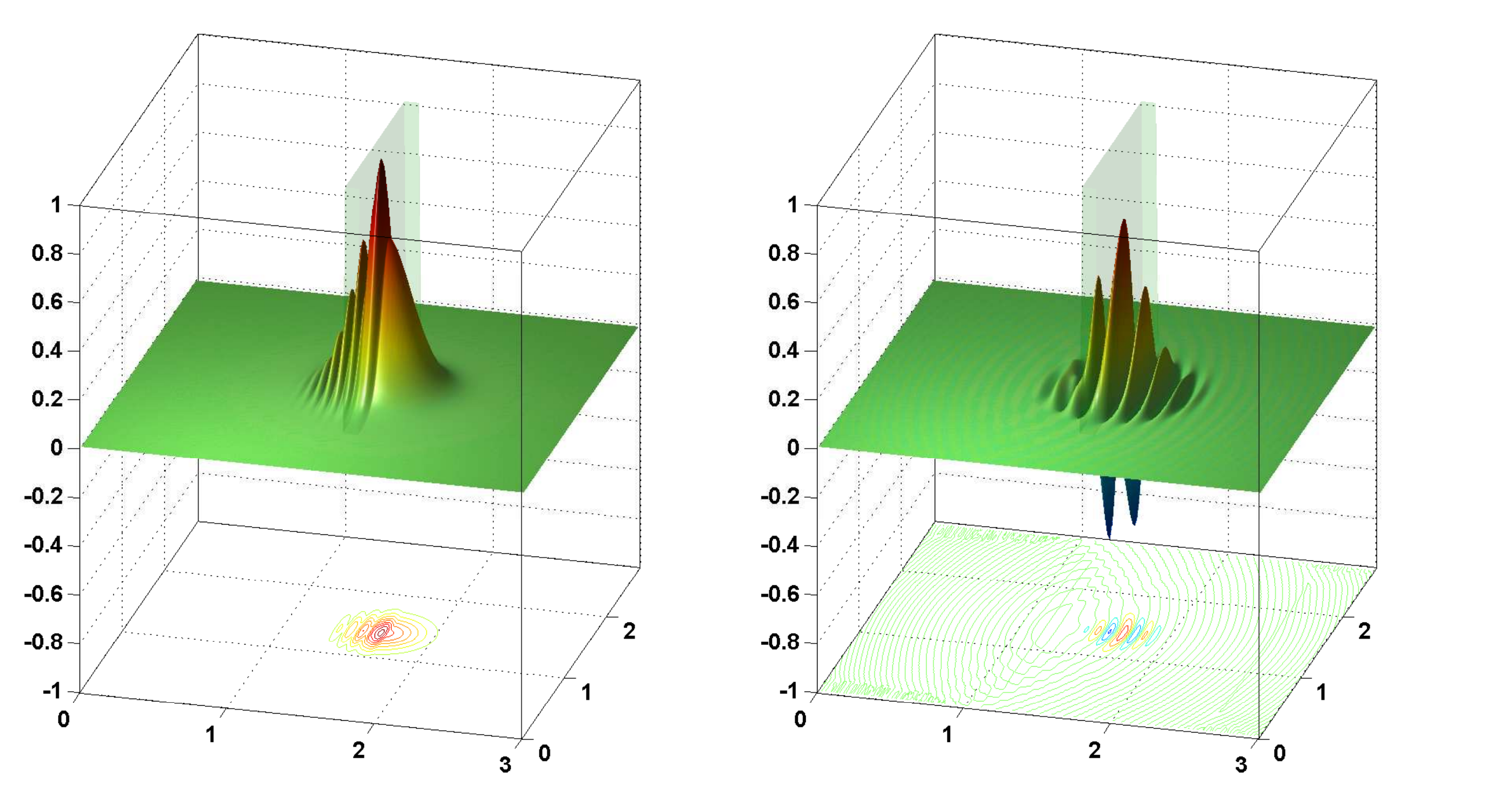}
    \vspace{0cm}\\
    \centerline{\small{$m=180$}}\\
    \includegraphics[scale=0.25]{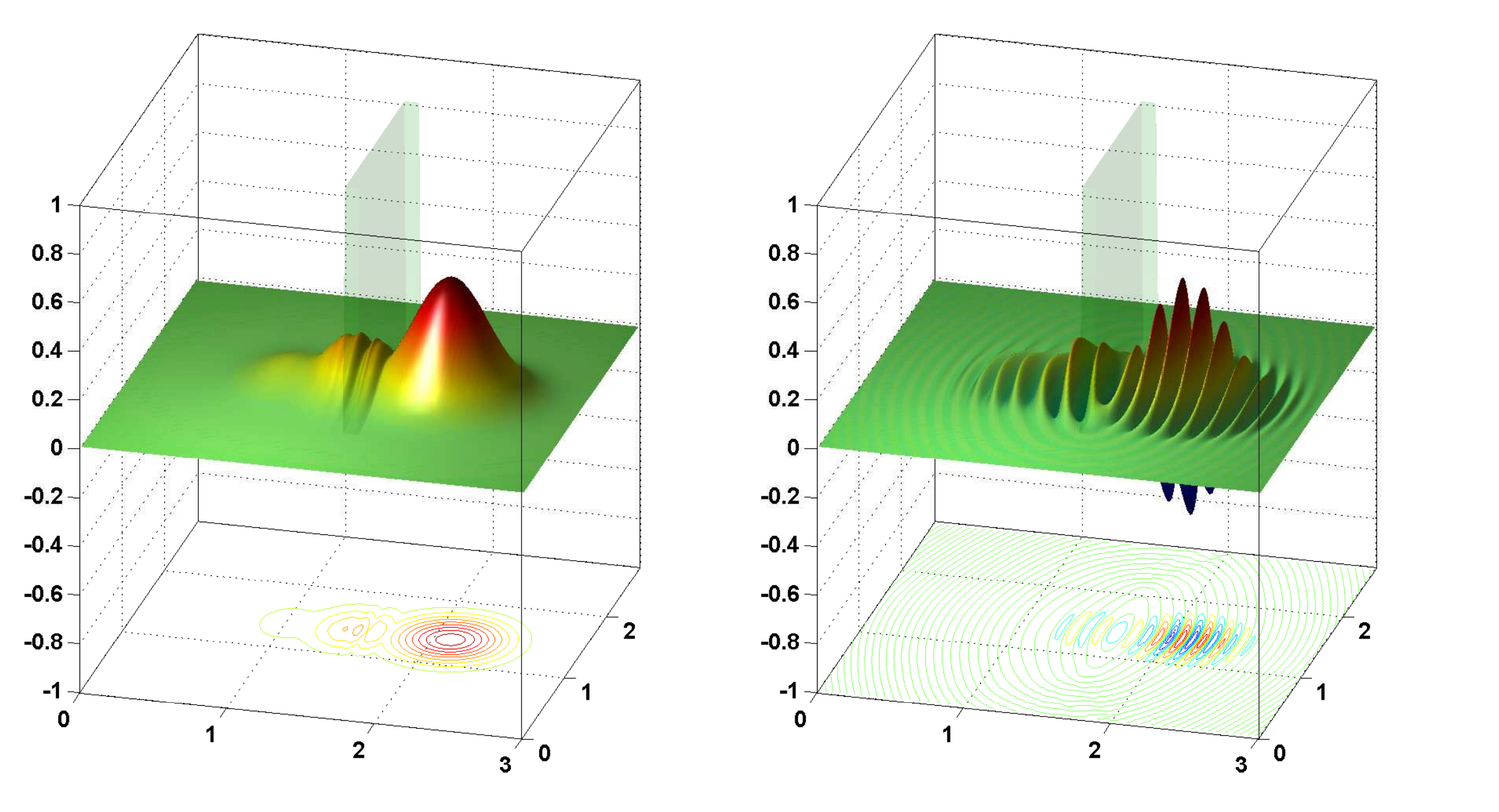}
    \vspace{0cm}\\
    \centerline{\small{$m=300$}}\\
\end{multicols}
\begin{multicols}{2}
    \includegraphics[scale=0.25]{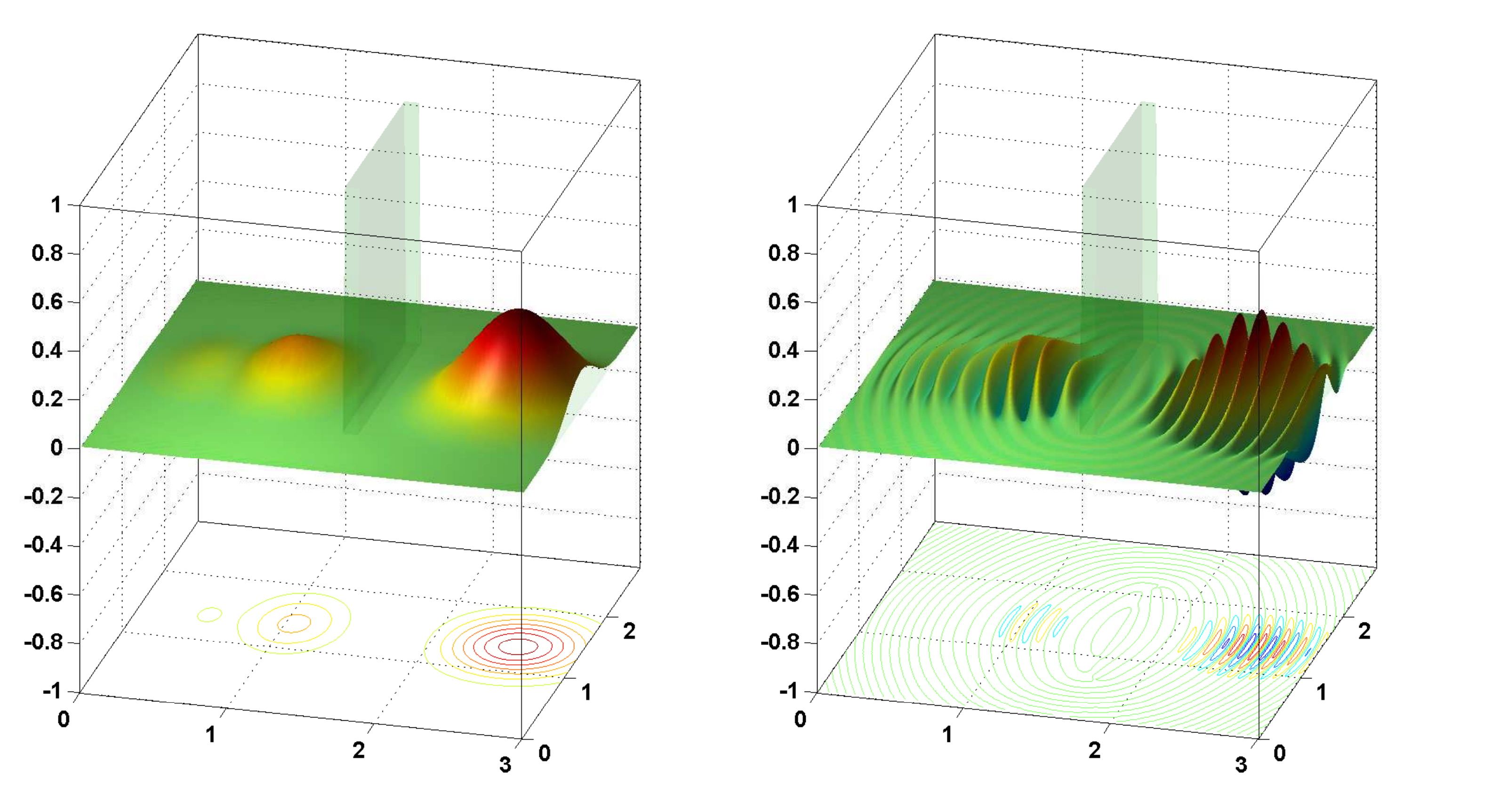}
    \vspace{0cm}\\
    \centerline{\small{$m=420$}}\\
    \includegraphics[scale=0.25]{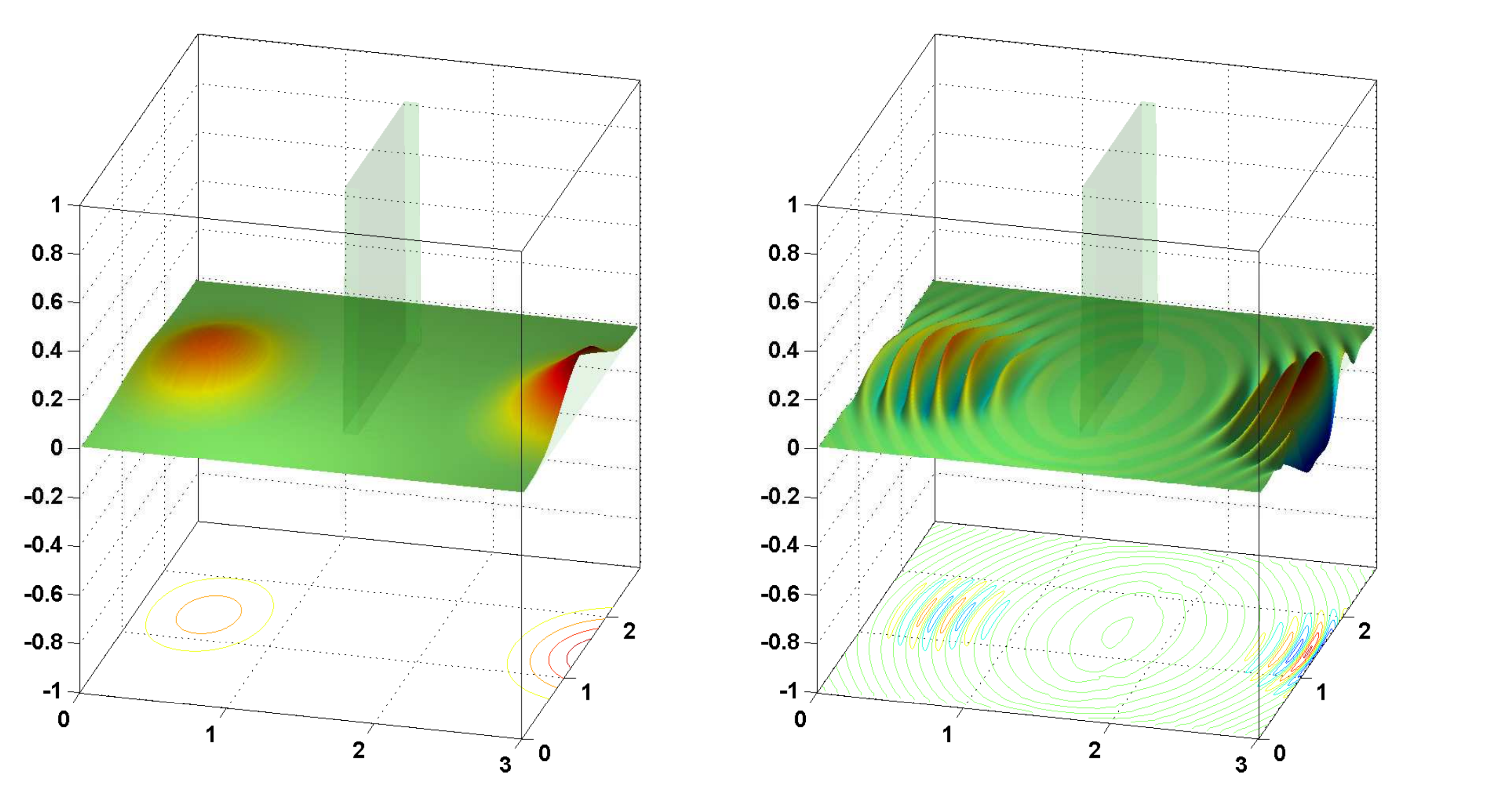}
    \vspace{0cm}\\
    \centerline{\small{$m=600$}}\\
    \end{multicols}
\caption{\small{Example B $(Q=1000)$. The modulus and the real part of the numerical solution $\Psi^m$
for $(J,K,M)=(1200,64,600)$}}
\label{SSP:EX22a:B:Q=1000:Solution}
\end{figure}
%
%
\par Third, let $Q = 4000$ be rather large. On Figure \ref{SSP:EX22a:B:Q=4000:Solution} the numerical solution is represented for the same time moments and the mesh. Here the main part of the wave is reflected from the barrier and then moves in the opposite direction along the $x$ axis.
\begin{figure}[ht]
\begin{multicols}{2}
    \includegraphics[scale=0.25]{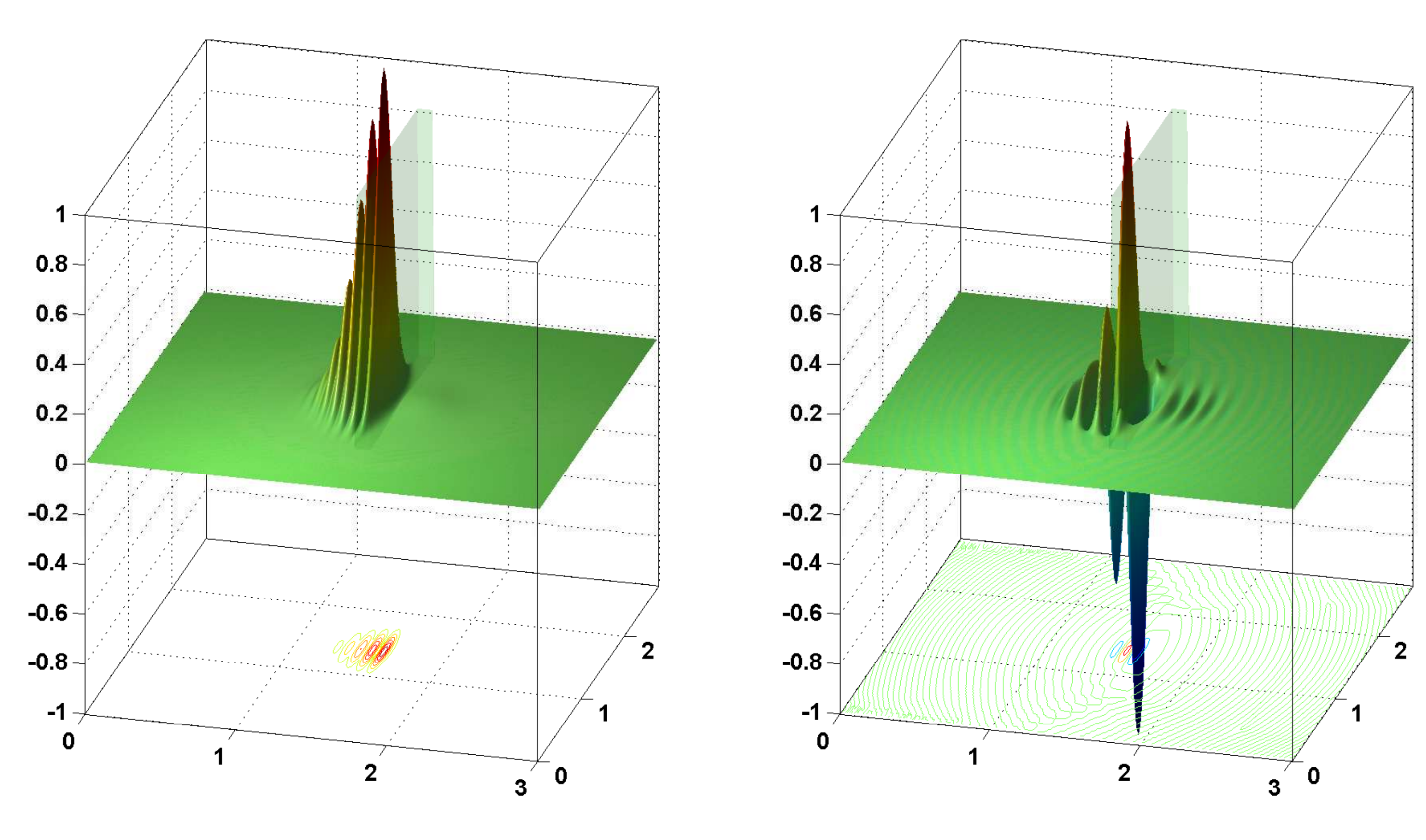}
    \vspace{0cm}\\
    \centerline{\small{$m=180$}}\\
    \includegraphics[scale=0.25]{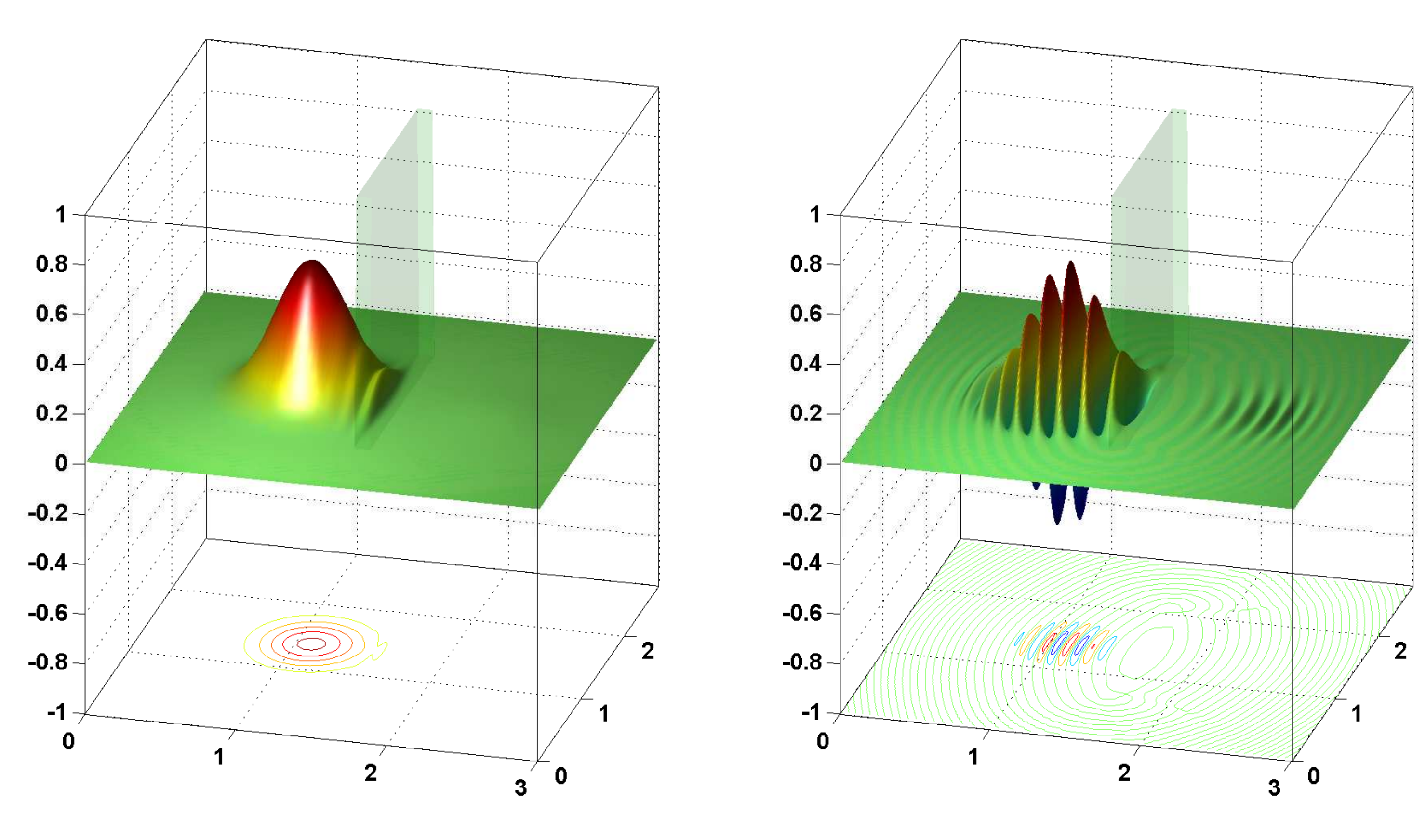}
    \vspace{0cm}\\
    \centerline{\small{$m=300$}}\\
\end{multicols}
\begin{multicols}{2}
    \includegraphics[scale=0.25]{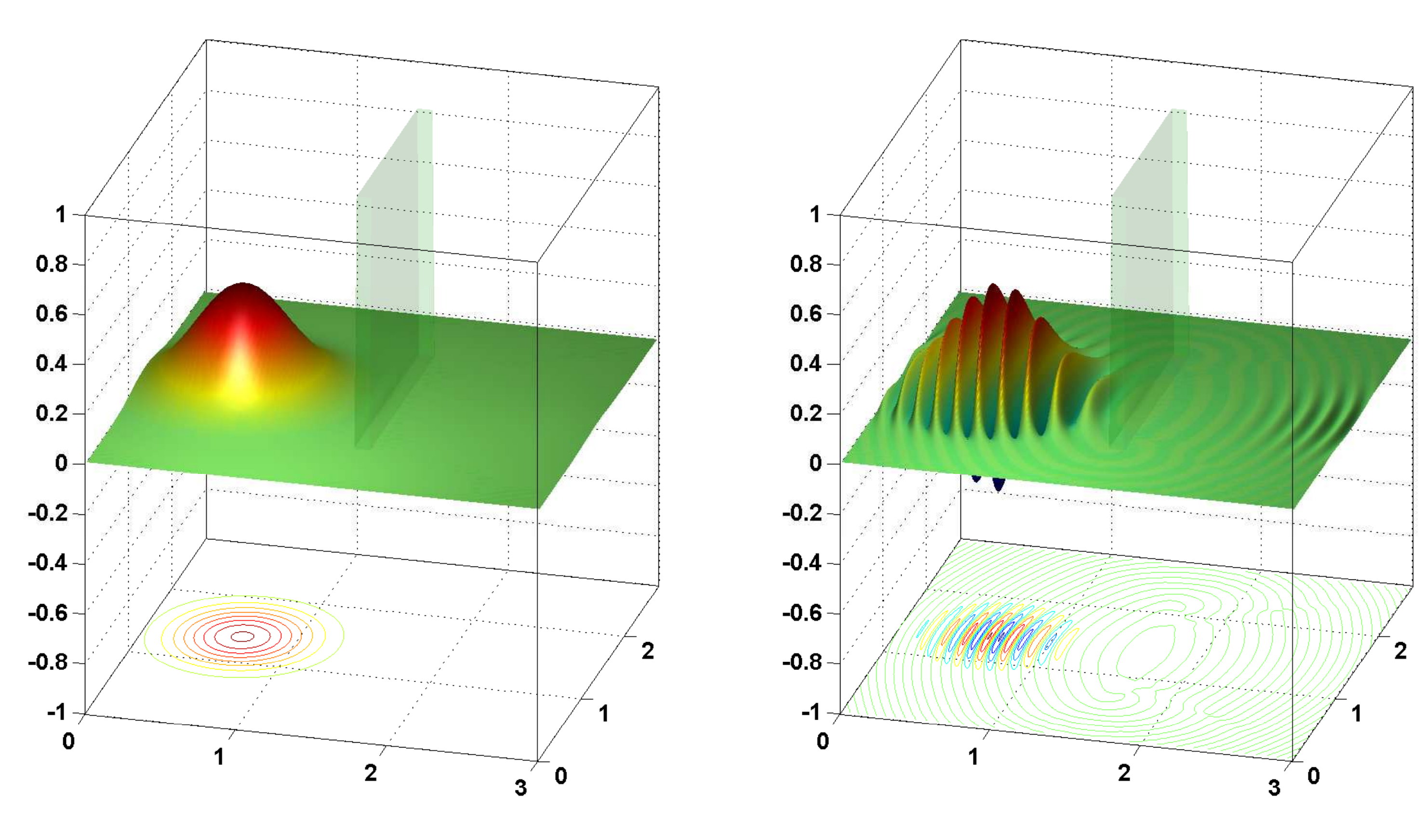}
    \vspace{0cm}\\
    \centerline{\small{$m=420$}}\\
    \includegraphics[scale=0.25]{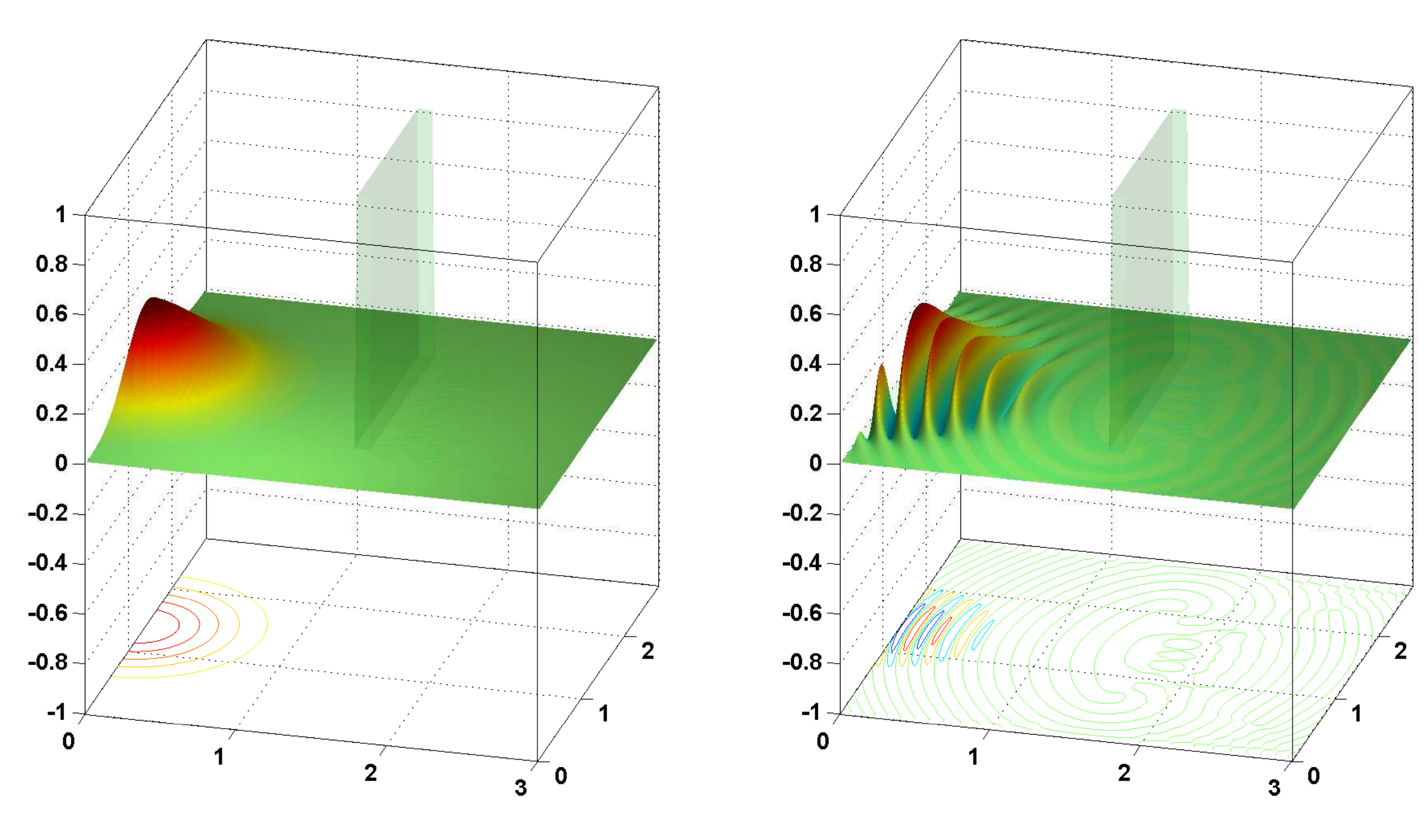}
    \vspace{0cm}\\
    \centerline{\small{$m=600$}}\\
\end{multicols}
\caption{\small{Example B ($Q=4000)$. The modulus and the real part of the numerical solution $\Psi^m$
for $(J,K,M)=(1200,64,600)$}}
\label{SSP:EX22a:B:Q=4000:Solution}
\end{figure}
%
\par To check the approximate solution in this case, we compute how the numerical solution changes when any of $J$, $K$ or $M$ increases twice, see Table \ref{SSP:tab:EX22a:4000}, where the corresponding absolute and relative errors in $C$ and $L_2$ norms are given.
The relative errors $E_{C,\,{\rm rel}}$ and $E_{{L^2},\,{\rm rel}}$
are defined as the maximal in time relative $C$ and $L_2$ mesh errors in space (in joint nodes).
One can see that all the errors are small enough.
\begin{table}\centering{
    \begin {tabular}{r<{\pgfplotstableresetcolortbloverhangright }@{}l<{\pgfplotstableresetcolortbloverhangleft }r<{\pgfplotstableresetcolortbloverhangright }@{}l<{\pgfplotstableresetcolortbloverhangleft }r<{\pgfplotstableresetcolortbloverhangright }@{}l<{\pgfplotstableresetcolortbloverhangleft }r<{\pgfplotstableresetcolortbloverhangright }@{}l<{\pgfplotstableresetcolortbloverhangleft }r<{\pgfplotstableresetcolortbloverhangright }@{}l<{\pgfplotstableresetcolortbloverhangleft }r<{\pgfplotstableresetcolortbloverhangright }@{}l<{\pgfplotstableresetcolortbloverhangleft }r<{\pgfplotstableresetcolortbloverhangright }@{}l<{\pgfplotstableresetcolortbloverhangleft }}%
\toprule \multicolumn {2}{c}{$J$}&\multicolumn {2}{c}{$K$}&\multicolumn {2}{c}{$M$}&\multicolumn {2}{c}{$E_{C}$}&\multicolumn {2}{c}{$E_{L^2}$}&\multicolumn {2}{c}{$E_{C,\,{\rm rel}}$}&\multicolumn {2}{c}{$E_{L^2,\,{\rm rel}}$}\\\midrule %
$2\,400$&$$&$64$&$$&$600$&$$&$8$&$.61\cdot 10^{-3}$&$5$&$.29\cdot 10^{-3}$&$2$&$.48\cdot 10^{-2}$&$2$&$.90\cdot 10^{-2}$\\%
$1\,200$&$$&$128$&$$&$600$&$$&$1$&$.38\cdot 10^{-2}$&$6$&$.37\cdot 10^{-3}$&$3$&$.56\cdot 10^{-2}$&$2$&$.82\cdot 10^{-2}$\\%
$1\,200$&$$&$64$&$$&$1\,200$&$$&$1$&$.22\cdot 10^{-2}$&$5$&$.30\cdot 10^{-3}$&$2$&$.82\cdot 10^{-2}$&$2$&$.33\cdot 10^{-2}$\\\bottomrule %
\end {tabular}%

\caption{Example B ($Q=4000$). The change in numerical solution when $J$, $K$ or $M$ increases twice}
\label{SSP:tab:EX22a:4000}}
\end{table}
\par We emphasize that on all Figures \ref{SSP:EX22a:B:Solution}, \ref{SSP:EX22a:B:Q=1000:Solution} and \ref{SSP:EX22a:B:Q=4000:Solution}, the last two graphs exhibit complete absence of the spurious reflections from the artificial left and right boundaries due to exploiting of the discrete TBCs there.
\smallskip\par \textbf{Example C}. We also treat the case of a very short barrier with $(c,d)=(\frac{Y}{2}-\frac{Y}{2^5},\frac{Y}{2}+\frac{Y}{2^5})=(1.3125,1.4875)$
and once again having the height $Q=1500$; this barrier looks like a column.
The numerical solution $\Psi^{m}$ is represented on Figure \ref{SSP:EX22a:C:Solution}, for the time moments $t_m=m\tau$, $m=180, 240, 300$ and $360$, together with the normalized barrier. We use the mesh with $(J,K,M)=(1200,512,600)$, i.e. for the same $J$ and $M$ as on the above figures but for notably larger $K$.
Now in contrast to the previous examples, the transmitted part of the wave package is separated into two pieces.
\begin{figure}[ht]
\begin{multicols}{2}
    \includegraphics[scale=0.25]{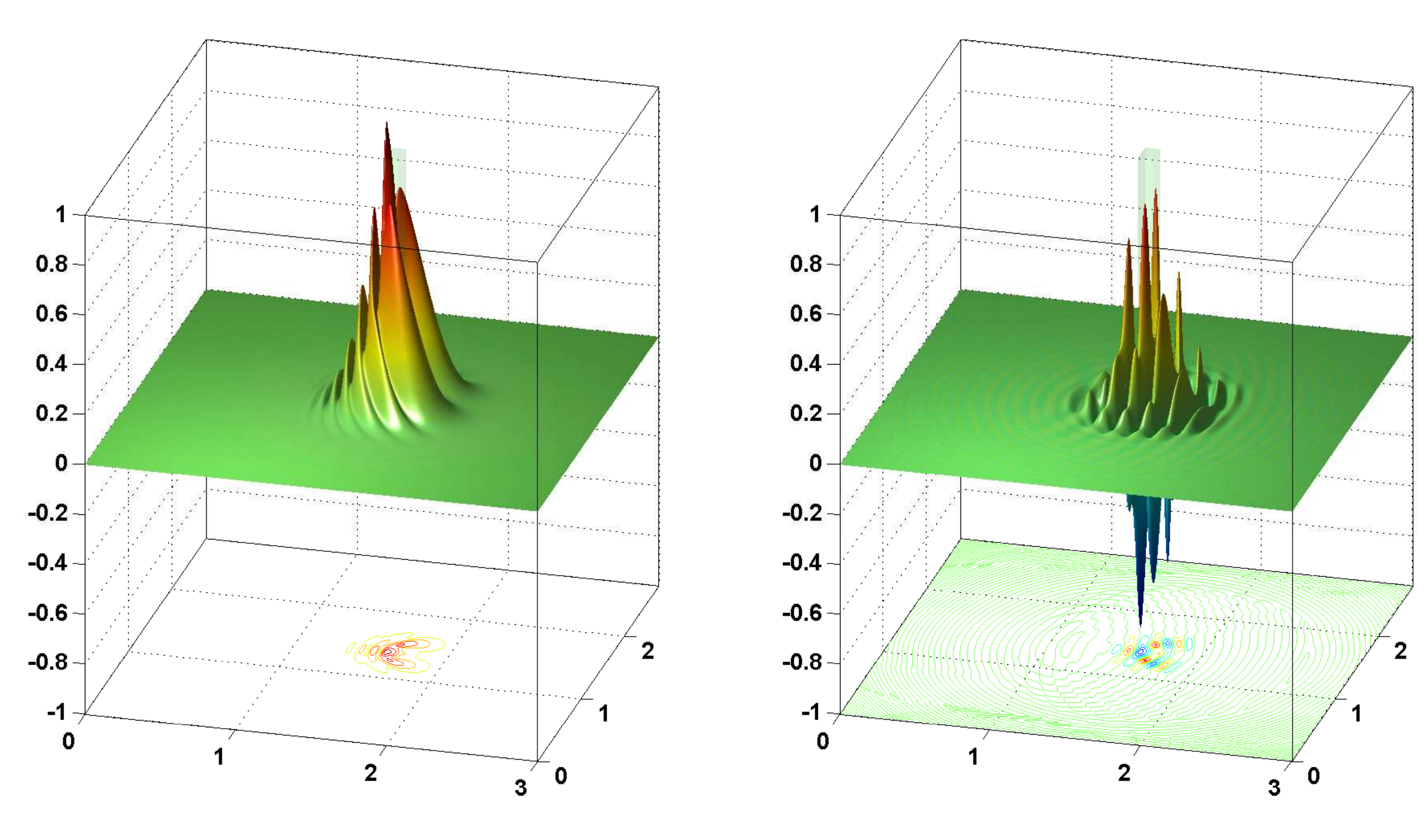}
    \vspace{0cm}\\
    \centerline{\small{$m=180$}}\\
    \includegraphics[scale=0.25]{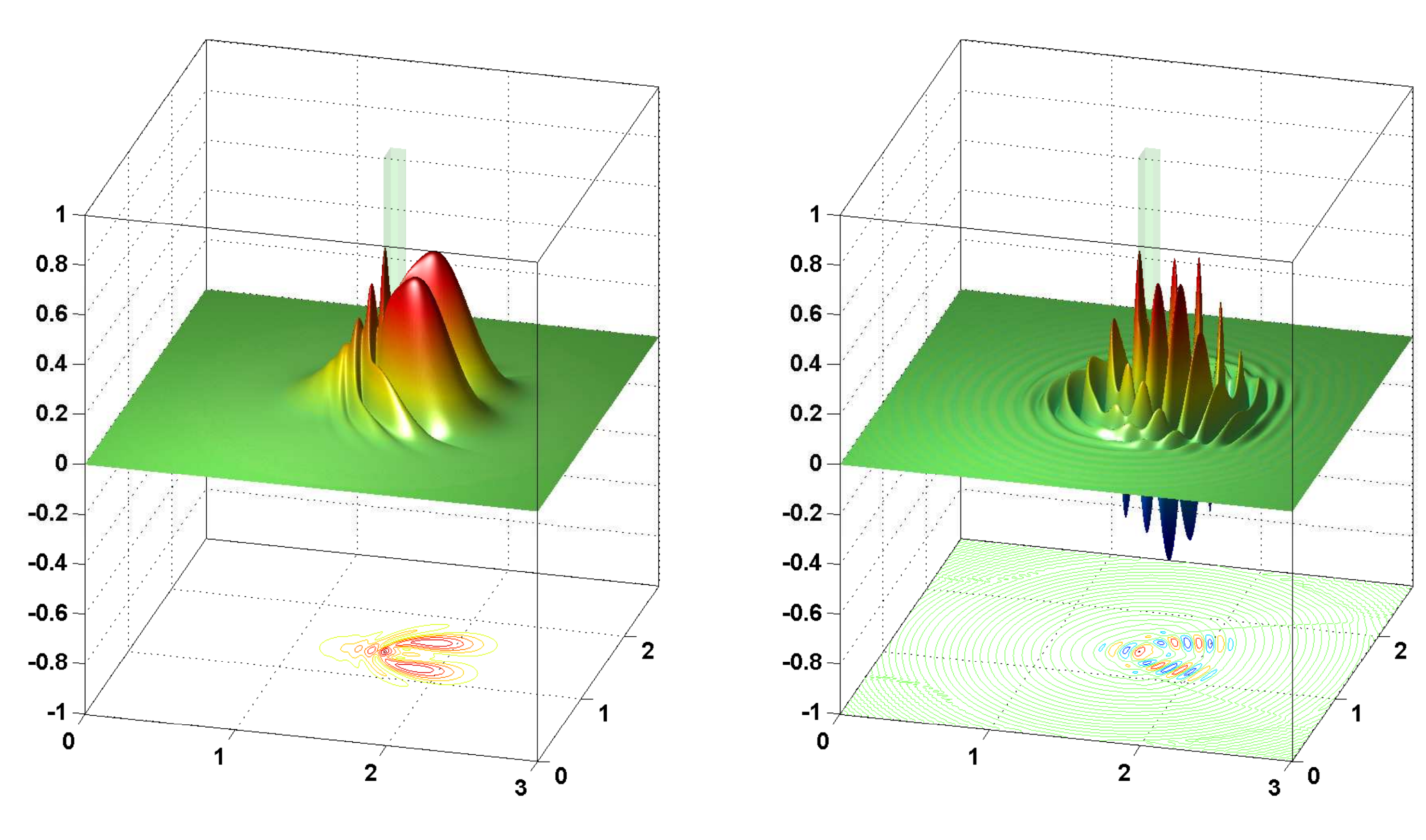}
    \vspace{0cm}\\
    \centerline{\small{$m=240$}}\\
\end{multicols}
\begin{multicols}{2}
    \includegraphics[scale=0.25]{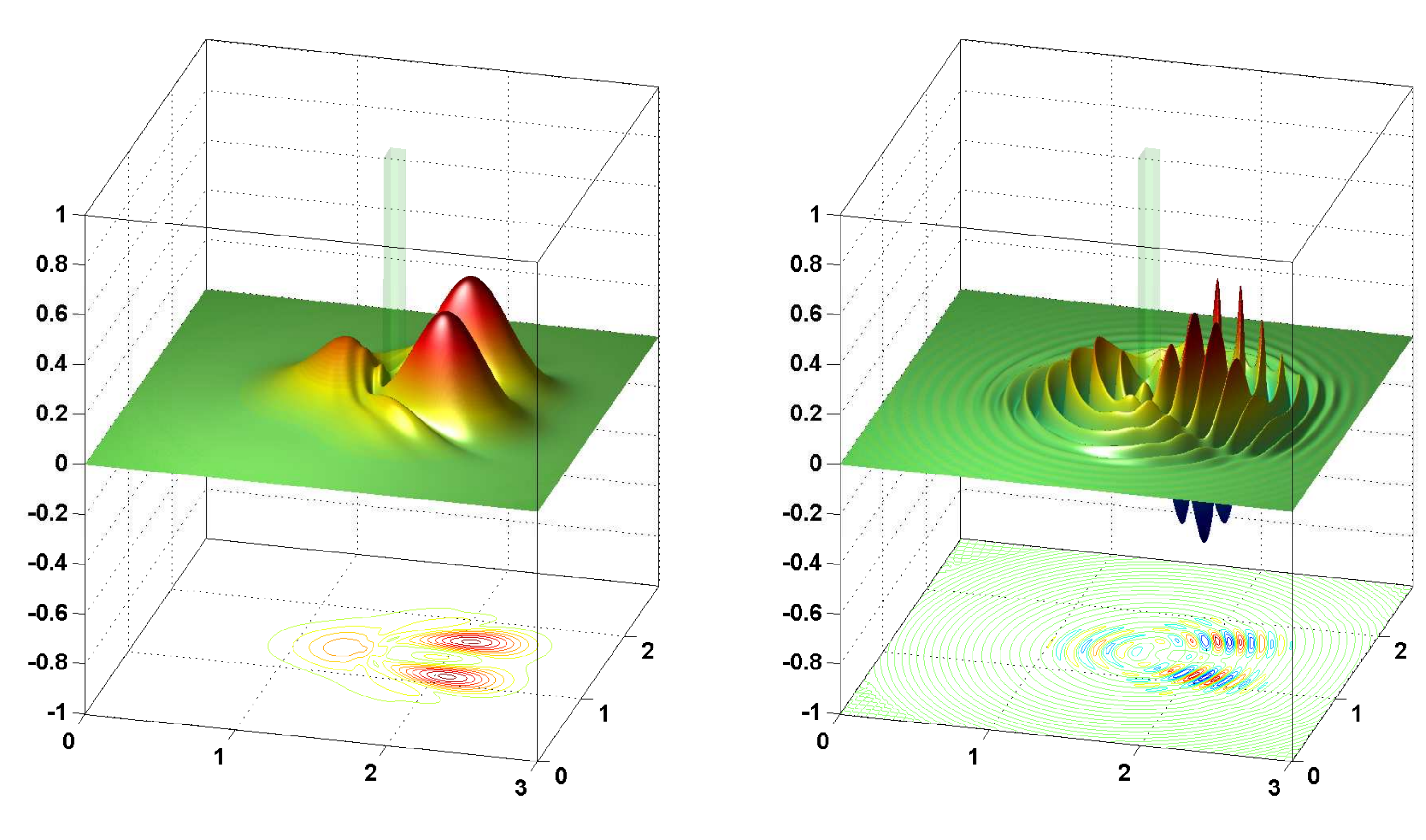}
    \vspace{0cm}\\
    \centerline{\small{$m=300$}}\\
    \includegraphics[scale=0.25]{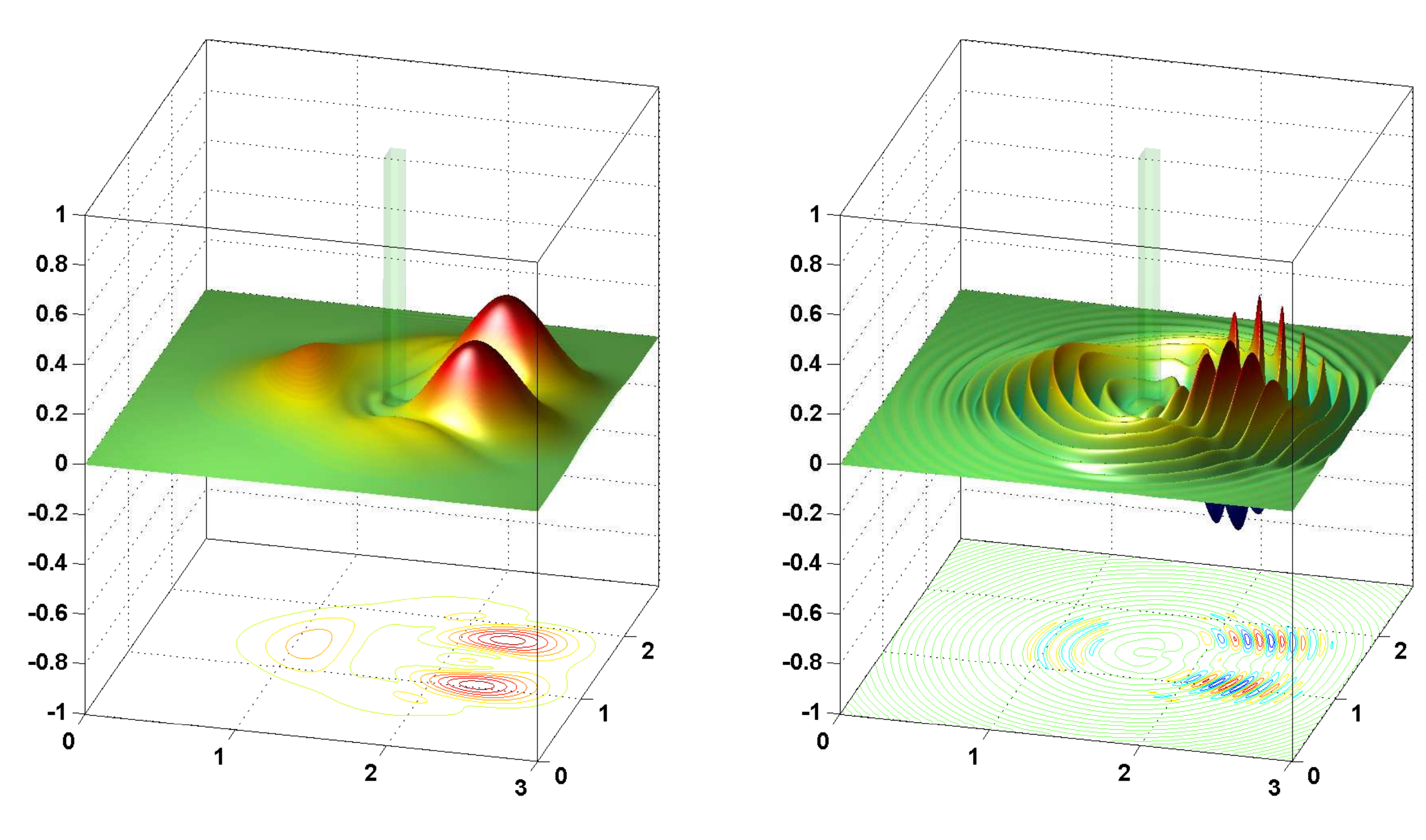}
    \vspace{0cm}\\
    \centerline{\small{$m=360$}}\\
\end{multicols}
\caption{\small{Example C. The modulus and the real part of the numerical solution $\Psi^m$
for $(J,K,M)=(1200,512,600)$}}
\label{SSP:EX22a:C:Solution}
\end{figure}
%
\par To check the approximate solution in this case, we compute how the numerical solution changes when $J$, $K$ or $M$ increases twice, see Table \ref{SSP:tab:EX22a:C}, where the corresponding absolute and relative errors are given. One can see that all of them are small enough once again.
\begin{table}\centering{
    \begin {tabular}{r<{\pgfplotstableresetcolortbloverhangright }@{}l<{\pgfplotstableresetcolortbloverhangleft }r<{\pgfplotstableresetcolortbloverhangright }@{}l<{\pgfplotstableresetcolortbloverhangleft }r<{\pgfplotstableresetcolortbloverhangright }@{}l<{\pgfplotstableresetcolortbloverhangleft }r<{\pgfplotstableresetcolortbloverhangright }@{}l<{\pgfplotstableresetcolortbloverhangleft }r<{\pgfplotstableresetcolortbloverhangright }@{}l<{\pgfplotstableresetcolortbloverhangleft }r<{\pgfplotstableresetcolortbloverhangright }@{}l<{\pgfplotstableresetcolortbloverhangleft }r<{\pgfplotstableresetcolortbloverhangright }@{}l<{\pgfplotstableresetcolortbloverhangleft }}%
\toprule \multicolumn {2}{c}{$J$}&\multicolumn {2}{c}{$K$}&\multicolumn {2}{c}{$M$}&\multicolumn {2}{c}{$E_{C}$}&\multicolumn {2}{c}{$E_{L^2}$}&\multicolumn {2}{c}{$E_{C,\,{\rm rel}}$}&\multicolumn {2}{c}{$E_{L^2,\,{\rm rel}}$}\\\midrule %
$2\,400$&$$&$512$&$$&$600$&$$&$1$&$.06\cdot 10^{-2}$&$5$&$.26\cdot 10^{-3}$&$3$&$.16\cdot 10^{-2}$&$2$&$.32\cdot 10^{-2}$\\%
$1\,200$&$$&$1\,024$&$$&$600$&$$&$9$&$.64\cdot 10^{-3}$&$6$&$.70\cdot 10^{-3}$&$3$&$.92\cdot 10^{-2}$&$4$&$.38\cdot 10^{-2}$\\%
$1\,200$&$$&$512$&$$&$1\,200$&$$&$8$&$.22\cdot 10^{-3}$&$4$&$.67\cdot 10^{-3}$&$2$&$.66\cdot 10^{-2}$&$2$&$.06\cdot 10^{-2}$\\\bottomrule %
\end {tabular}%

\caption{Example C.  The change in numerical solution when $J$, $K$ or $M$ increases twice}
\label{SSP:tab:EX22a:C}}
\end{table}
\smallskip\par  We call attention to the essentially more complicated behavior of the real part of the solution compared to its modulus in all Examples A-C. Comparing $C$ and $L^2$ errors, one can see that though $C$ errors are mainly larger, their behavior is rather similar that is not so obvious a priori taking into account the oscillatory type of the solutions in space and time.
\par In general, the above practical error analysis indicates the good error properties of the splitting in potential scheme.
\par Finally, note that clearly both the rectangular form of the barrier and the specific choice of the initial function are inessential to apply efficiently the splitting method.
\smallskip\par
\textbf{Acknowledgments}
\smallskip\par
The paper has been initiated during the visit of A. Zlotnik in summer 2011 to the the D\'epar\-te\-ment de Physique Th\'eorique et Appliqu\'ee, CEA/DAM/DIF Ile de France (Arpajon), which he thanks for hospitality.
The study is carried out by him within The National Research University Higher School of Economics' Academic Fund Program in 2012-2013, research grant No. 11-01-0051.
Both A. Zlotnik and I. Zlotnik are supported by the Ministry of Education and Science of the Russian Federation (contract 14.B37.21.0864) and the Russian Foundation for Basic Research, project 12-01-90008-Bel.


\begin{thebibliography}{99}
\bibitem{AABES08}
X. Antoine, A. Arnold, C. Besse, M. Ehrhardt and A. Sch\"adle,
A review of transparent and artificial boundary conditions techniques for linear and nonlinear Schr\"odinger equations.
{\sl Commun. Comp. Phys.} {\bf 4} (4) (2008) 729-796.
\bibitem{ABM04}
X. Antoine, C. Besse and V. Mouysset,
Numerical schemes for the simulation of the two-dimensional Schr\"o\-dinger equation using non-reflecting boundary conditions.
{\sl Math. Comp.} {\bf 73} (2004) 1779-1799.
\bibitem{AES03}
A. Arnold, M. Ehrhardt and I. Sofronov,
Discrete transparent boundary conditions for the Schr\"o\-dinger equation: fast calculations, approximation and stability.
{\sl Comm. Math. Sci.} {\bf 1} (2003) 501-556.
%
\bibitem{BGG84}
J.F. Berger, M. Girod and D. Gogny,
 Microscopic analysis of collective dynamics in low energy fission,
{\em Nuclear Physics} {\bf A 428} 23c-36c (1984).
%
\bibitem{BGG91}
J.-F. Berger, M. Girod and D. Gogny,
Time-dependent quantum collective dynamics applied to nuclear fission.
{\sl Comp. Phys. Comm.} {\bf 63} (1991) 365-374.
\bibitem{BM00}
S. Blanes and P.C. Moan,
Splitting methods for the time-dependent Schr\"o\-dinger equation.
{\sl Phys. Lett. A} {\bf 265} (2000) 35-42.
%
\bibitem{CBGW84}
C.R. Chinn, J.F. Berger, D. Gogny and M.S. Weiss,
 Limits on the lifetime of the shape isomer of $^{238}U$,
{\em Physical Review} {\bf C 45} (1984) 1700--1708.
%
\bibitem{DiM97}
L. Di Menza,
Transparent and absorbing boundary conditions for the Schr\"odinger equation in a bounded domain.
{\sl Numer. Funct. Anal. and Optimiz.} {\bf 18} (1997) 759-775.
\bibitem{DZ06}
B. Ducomet and A. Zlotnik,
On stability of the Crank-Nicolson scheme with approximate transparent boundary conditions for the Schr\"odinger equation. Part I.
{\sl Comm. Math. Sci.} {\bf 4} (2006) 741-766.
\bibitem{DZ07}
B. Ducomet and A. Zlotnik,
On stability of the Crank-Nicolson scheme with approximate transparent boundary conditions for the Schr\"odinger equation. Part II.
{\sl Comm. Math. Sci.} {\bf 5} (2007) 267-298.
\bibitem{DZR13}
B. Ducomet, A. Zlotnik and A. Romanova,
A splitting higher order scheme with discrete transparent boundary conditions for the Schr\"odinger equation in a semi-infinite parallelepiped.
{\it In preparation}
\bibitem{DZZ09}
B. Ducomet, A. Zlotnik and I. Zlotnik,
On a family of finite-difference schemes with discrete transparent boundary conditions for a generalized 1D Schr\"odinger equation,
{\sl Kinetic and Related Models}, {\bf 2} (2009), 151-179.
\bibitem{EA01}
M. Ehrhardt and A. Arnold,
Discrete transparent boundary conditions for the Schr\"odinger equation.
{\sl Riv. Mat. Univ. Parma} {\bf 6} (2001) 57-108.
\bibitem{GX11}
Z. Gao and S. Xie, Fourth-order alternating direction implicit compact finite difference schemes for two-dimensional Schr\"odinger equations.
{\sl Appl. Numer. Math.} {\bf 61} (2011) 593-614.
\bibitem{G11}
L. Gauckler, Convergence of a split-step Hermite method for Gross-Pitaevskii equation.
{\sl IMA J. Numer. Anal.} {\bf 31} (2011) 396-415.
\bibitem{GBCG05}
H. Goutte, J.-F. Berger, P. Casoly and D. Gogny,
Microscopic approach of fission dynamics applied to fragment kinetic energy and mass distribution in $\,^{238}U$.
{\sl Phys. Rev. C} {\bf 71} (2) (2005) 4316 (1-13).
%
\bibitem{H74}
H. Hofmann,
 Quantummechanical treatment of the penetration through a two-dimensional fission barrier,
{\em Nuclear Physics} {\bf A 224} 116 (1974).
%
\bibitem{L08a}
C. Lubich, On splitting methods for Schr\"odinger-Poisson and cubic nonlinear Schr\"odinger equations.
{\sl Math. Comput.} {\bf 77} (264) (2008) 2141-2153.
\bibitem{L08}
C. Lubich, From quantum to classical molecular dynamics. Reduced models and numerical analysis. EMS: Z\"{u}rich, 2008.
\bibitem{NT09}
C. Neuhauser and M. Thalhammer,
On the convergence of splitting methods for linear evolutionary Schr\"odinger equations involving an unbounded potential.
{\sl BIT Numer Math.} {\bf 49} (2009) 199-215.
\bibitem{RHR78}
P. Ring, H. Hassman and J.O. Rasmussen,
 On the treatment of a two-dimensional fission model with complex trajectories,
{\em Nuclear Physics} {\bf A 296} 50 (1978).
\bibitem{RS80}
P. Ring and P. Schuck,
{\it The nuclear many-body problem,}
Springer-Verlag, New York, Heidelberg, Berlin (1980).
\bibitem{RDNPS78}
S.G. Rohozinski, J. Dobaczewski, B. Nerlo-Pomorska, K. Pomorski and J. Srebny,
 Microscopic dynamic calculations of collective states in Xenon and Barium isotopes,
{\em Nuclear Physics} {\bf A 292} 66 (1978).
\bibitem{Sch02a}
A. Sch\"adle, Non-reflecting boundary conditions for the two-dimensional Schr\"o\-dinger equation.
{\sl Wave Motion} {\bf 35} (2002) 181-188.
\bibitem{SA08}
M. Schulte and A. Arnold,
Discrete transparent boundary conditions for the Schr\"o\-dinger equation, a compact higher order scheme.
{\sl Kinetic and Related Models} {\bf 1} (1) (2008) 101-125.
\bibitem{S68}
G. Strang, On the construction and comparison of difference scheme.
{\sl SIAM J. Numer. Anal.} {\bf 5} (1968) 506-517.
\bibitem{SZ04}
J. Szeftel, Design of absorbing boundary conditions for Schr\"o\-dinger equations in $\mathbb{R}^d$.
{\sl SIAM J. Numer. Anal.} {\bf 42} 2004 (4) 1527-1551.
\bibitem{Ya71}
N.N. Yanenko, The method of fractional steps: solution of problems of mathematical physics in several variables. Springer: New York, 1971.
\bibitem{ZaZ98}
S.B. Zaitseva and A.A. Zlotnik, Error analysis in $L_2(Q)$ for symmetrized locally one-dimensional methods for the heat equation.
{\sl Russ. J. Numer. Anal. Math. Model.} {\bf 13} 1998 (1) 69-91.
\bibitem{ZaZ99}
S.B. Zaitseva and A.A. Zlotnik, Sharp error analysis of vector splitting methods for the heat equation.
{\sl Comp. Maths. Math. Phys.} {\bf 39} 1999 (3) 448-467.
\bibitem{Z91}
A.A. Zlotnik,
Some finite-element and finite-difference methods for solving mathematical physics problems with non-smooth data in $n$-dimensional cube.
{\sl Sov. J. Numer. Anal. Math. Modelling} {\bf 6} (1991) 421-451.
\bibitem{ZI06}
A. Zlotnik and S. Ilyicheva,
Sharp error bounds for a symmetrized locally 1D method for solving the 2D heat equation.
{\sl Comp. Meth. Appl. Math.} {\bf 6} (1) (2006) 94-114.
\bibitem{ZR13}
A. Zlotnik and A. Romanova,
A Numerov-Crank-Nicolson-Strang scheme with discrete transparent boundary conditions for the Schr\"odinger equation on a semi-infinite strip. Submitted, see also http://arxiv.org/abs/1307.5398.
\bibitem{ZZ11}
A.A. Zlotnik and I.A. Zlotnik, Family of finite-difference schemes with transparent boundary conditions for the nonstationary Schr\"o\-dinger equation in a semi-infinite strip. {\sl Dokl. Math.} {\bf 83} (1) (2011) 12-18.
\bibitem{IZ10}
I.A. Zlotnik,
Computer simulation of the tunnel effect.
{\sl Moscow Power Engin. Inst. Bulletin} (6) (2010) 10-28 (in Russian).
\bibitem{IZ11}
I.A. Zlotnik,
Family of finite-difference schemes with approximate transparent boundary conditions for the generalized nonstationary Schr\"o\-dinger equation in a semi-infinite strip.
{\sl Comput. Maths. Math. Phys.} {\bf 51} (3) (2011) 355-376.
\end{thebibliography}
\end{document}